\documentclass{amsart}

\usepackage{amssymb,amscd,amsthm,verbatim,fancyhdr}
\usepackage{mathrsfs}

\usepackage[colorlinks=true,citecolor=blue,linkcolor=magenta]{hyperref}

\begin{document}
\numberwithin{equation}{section}
\newtheorem{thm}{Theorem}
\newtheorem{lemma}{Lemma}[section]
\newtheorem{clm}[lemma]{Claim}
\newtheorem{remark}[lemma]{Remark}
\newtheorem{definition}[lemma]{Definition}
\newtheorem{cor}[lemma]{Corollary}
\newtheorem{prop}[lemma]{Proposition}
\newtheorem{statement}[lemma]{Statement}


\newcommand{\hdt}{{\dot{\mathrm{H}}^{1/2}}}
\newcommand{\hdtr}{{\dot{\mathrm{H}}^{1/2}(\mathbb{R}^3)}}
\newcommand{\R}{\mathbb{R}}
\newcommand{\ei}{\mathrm{e}^{it\Delta}}
\newcommand{\ltrt}{{L^3(\mathbb{R}^3)}}
\newcommand{\ldrd}{{L^d(\mathbb{R}^d)}}
\newcommand{\lprd}{{L^p(\mathbb{R}^d)}}
\newcommand{\lt}{{L^3}}
\newcommand{\ld}{{L^d}}
\newcommand{\lp}{{L^p}}
\newcommand{\rt}{\mathbb{R}^3}
\newcommand{\rd}{\mathbb{R}^d}
\newcommand{\X}{\mathfrak{X}}
\newcommand{\F}{\mathfrak{F}}
\newcommand{\hdhalf}{{\dot H^\frac{1}{2}}}
\newcommand{\hdthalf}{{\dot H^\frac{3}{2}}}
\newcommand{\hdo}{\dot H^1}
\newcommand{\rthmiz}{\R^3\times(-\infty,0)}
\newcommand{\q}[2]{{#1}_{#2}}
\renewcommand{\t}{\theta}
\newcommand{\lxt}[2]{L_{x,\,t}^{#1}}
\newcommand{\rr}{\sqrt{x_1^2+x_2^2}}
\newcommand{\ve}{\varepsilon}
\newcommand{\hdhrt}{\dot H^\frac{1}{2}(\mathbb{R}^3)}
\renewcommand{\P}{\mathbb{P} }
\newcommand{\RR}{\mathcal{R} }
\newcommand{\TT}{\overline{T} }
\newcommand{\e}{\epsilon }
\newcommand{\D}{\Delta }
\renewcommand{\d}{\delta }
\renewcommand{\l}{\lambda }
\newcommand{\To}{\TT_1 }
\newcommand{\ukt}{u^{(KT)} }
\newcommand{\ttil}{\tilde{T_1} }
\newcommand{\etl}{e^{t\Delta} }
\newcommand{\et}{\mathscr{E}_T }
\newcommand{\se}{\mathscr{E}}
\newcommand{\ft}{\mathscr{F}_T }
\newcommand{\eti}{\mathscr{E}^{\infty}_T }
\newcommand{\fti}{\mathscr{F}^{\infty}_T }
\newcommand{\xjn}{x_{j,n}}
\newcommand{\xjpn}{x_{j',n}}
\newcommand{\ljn}{\l_{j,n} }
\newcommand{\ljpn}{\l_{j',n} }
\newcommand{\lkn}{\l_{k,n} }
\newcommand{\voj}{U_{0,j} }
\newcommand{\uoj}{U_{0,j} }
\newcommand{\voo}{U_{0,1} }
\newcommand{\uon}{u_{0,n} }
\newcommand{\N}{\mathbb{N} }
\newcommand{\Z}{\mathbb{Z} }
\newcommand{\E}{\mathscr{E} }
\newcommand{\tu}{\tilde{u} }
\newcommand{\tU}{\tilde{U} }
\newcommand{\etj}{{\E_{T^*_j}}}
\newcommand{\B}{\mathscr{B}}
\newcommand{\tujn}{\tilde{U}_{j,n}}
\newcommand{\tujpn}{\tilde{U}_{j',n}}
\newcommand{\soj}{\sum_{j=1}^J}
\newcommand{\soi}{\sum_{j=1}^\infty}
\newcommand{\hnj}{H_{n,J}}
\newcommand{\enj}{e_{n,J}}
\newcommand{\pnj}{p_{n,J}}
\newcommand{\lfoi}{L^5_{(0,\infty)}}
\newcommand{\lfhoi}{L^{{5/2}}_{(0,\infty)}}
\newcommand{\wnj}{w_n^J}
\newcommand{\rnj}{r_n^J}
\newcommand{\lfij}{{L^5_{I_j}}}
\newcommand{\lfiz}{{L^5_{I_0}}}
\newcommand{\lfio}{{L^5_{I_1}}}
\newcommand{\lfit}{{L^5_{I_2}}}
\newcommand{\lfhij}{{L^{{5/2}}_{I_j}}}
\newcommand{\lfhiz}{{L^{{5/2}}_{I_0}}}
\newcommand{\lfhio}{{L^{{5/2}}_{I_1}}}
\newcommand{\lfhit}{{L^{{5/2}}_{I_2}}}
\newcommand{\lfi}{{L^5_{I}}}
\newcommand{\lfhi}{{L^{{5/2}}_{I}}}
\newcommand{\doh}{D^\frac{1}{2}}
\newcommand{\tjn}{t_{j,n}}
\newcommand{\tjpn}{t_{j',n}}
\newcommand{\wnlj}{w_n^{l,J}}
\renewcommand{\O}{\mathcal{O}}
\newcommand{\bes}{{\dot B^{s_p}_{p,q}}}
\newcommand{\bespp}{{\dot B^{s_p}_{p,p}}}
\newcommand{\besrr}{{\dot B^{s_r}_{r,r}}}
\newcommand{\bespq}{{\dot B^{s_p}_{p,q}}}
\newcommand{\besqq}{{\dot B^{s_q}_{q,q}}}
\newcommand{\besa}{{\dot B^{s_{a}}_{a,b}}}
\newcommand{\besb}{{\dot B^{s_{p}}_{p,q}}}
\newcommand{\besr}{{\dot B^{s_{b}+\frac{2}{\rho}}_{b,r}}}
\newcommand{\tn}{{\tilde \|}}
\newcommand{\tp}{{\tilde{\phi}}}
\newcommand{\xto}{\xrightarrow[n\to\infty]{}}
\newcommand{\LL}{\Lambda}
\newcommand{\binf}{{\dot B^{-{d/p}}_{\infty,\infty}}}
\newcommand{\brq}{{\dot B^{s_{p,r}}_{r,q}}}
\newcommand{\bpp}{{\dot B^{s_{p}}_{p,p}}}
\newcommand{\bqq}{{\dot B^{s_{q}}_{q,q}}}
\newcommand{\trnj}{{ \tilde{r}_n^J}}
\newcommand{\vn}[1]{v^{(#1)}}
\newcommand{\vl}[1]{v_L^{(#1)}}
\newcommand{\tvl}[1]{{\tilde v}_L^{(#1)}}
\newcommand{\bvl}[1]{{\bar v}_L^{(#1)}}
\newcommand{\tvn}[1]{{\tilde v}^{(#1)}}
\newcommand{\bn}[3]{B^{(#1)}(#2,#3)}
\newcommand{\lr}{\mathscr{L}^r}
\renewcommand{\ln}[1]{\mathscr{L}^{#1}}
\newcommand{\tl}{\mathcal{L}}
\newcommand{\cl}{\mathscr{L}}
\newcommand{\ck}{\mathscr{K}}
\newcommand{\ol}{\widehat{\mathscr{L}}}

\def\longformule#1#2{
\displaylines{ \qquad{#1} \hfill\cr \hfill {#2} \qquad\cr } }
\def\sumetage#1#2{
\sum_{\scriptstyle {#1}\atop\scriptstyle {#2}} }
\def\longformule#1#2{ \displaylines{\qquad{#1} \hfill\cr \hfill {#2} \qquad\cr } }
\def\L {\mathscr L}
\title[Blow-up of Navier-Stokes critical Besov norms]{Blow-up of critical Besov norms at a\\ potential Navier-Stokes singularity}
\author{Isabelle Gallagher}
\address{Institut de Math\'ematiques de Jussieu-Paris Rive Gauche UMR
  CNRS 7586\\
Universit\'e Paris-Diderot\\
B\^atiment Sophie Germain \\
Case 7012
75205 PARIS Cedex 13, FRANCE}
\email{gallagher@math.univ-paris-diderot.fr}

\author{Gabriel S. Koch}
\address{Department of Mathematics\\ University of Sussex\\
Brighton, BN1 9QH, United Kingdom}
\email{g.koch@sussex.ac.uk}
\author{Fabrice Planchon}
\address{Laboratoire J. A. Dieudonn\'e, UMR CNRS 7351\\
Universit\'e Nice Sophia-Antipolis\\
Parc Valrose\\
06108 Nice Cedex 02\\
FRANCE }
\email{fabrice.planchon@unice.fr}
  \thanks{The first  author was partially supported by
the A.N.R grant ANR-12-BS01-0013-01 ``Harmonic Analysis at its Boundaries", as well as  the Institut Universitaire de France.
The second author was partially supported by the
EPSRC grant EP/M019438/1, ``Analysis of the Navier-Stokes regularity
problem".  The third author was partially supported by A.N.R. grant GEODISP as
well as the Institut Universitaire de France}
\date{}
\maketitle
\begin{abstract}
We prove that if an initial datum to the incompressible Navier-Stokes equations  in any critical Besov space~$\dot B^{-1+\frac 3p}_{p,q}(\R^3)$, with~$3 <p,q< \infty$, gives rise to
a strong solution with a singularity at a finite time~$T>0$, then the norm of the solution in that Besov space becomes unbounded at time~$T$.
This result, which treats all critical Besov spaces where local existence is known,  generalizes the result of  Escauriaza, Seregin and \v Sver\'
ak (Uspekhi Mat. Nauk 58(2(350)):3-44, 2003) concerning suitable weak
solutions blowing up in~$L^3(\R^3)$.  Our proof uses profile
decompositions and is based on our previous work
(Math. Ann. 355(4):1527--1559, 2013) which provided an alternative
proof of the $L^3(\R^3)$ result.  For very large values of $p$,   an
iterative method, which may be of independent interest, enables us to use some techniques from the~$L^3(\R^3)$ setting.
\end{abstract}
\tableofcontents

\pagebreak

\section{Introduction}
\subsection{The Navier-Stokes blow-up problem in critical spaces}
Consider the following  Navier-Stokes equations governing the velocity vector field $u(\cdot,t) : \R^3 \to \R^3$ (and scalar pressure $\pi$) of an incompressible, viscous, homogeneous fluid:
$$
{\rm(NS)} \quad  \left\{
\begin{array}{rcl}
\partial_t u + (u\cdot \nabla) u -\Delta u&=& -\nabla \pi  \quad \mbox{in} \:
\R^3 \times (0,T) \\
\mbox{div} \, u&=& 0\\
u_{|t=0} &=& u_{0}\, .
\end{array}
\right.
$$
The spaces $X$ appearing in the chain of continuous embeddings (see Definition~\ref{d1} below)
\begin{equation}\label{embed}
\dot H^{\frac{1}2}(\rt) \hookrightarrow L^3 (\rt) \hookrightarrow \dot B^{-1+\frac 3p}_{p,q}(\rt) \hookrightarrow \dot B^{-1+\frac 3{p'}}_{p',q'}(\rt)
\end{equation}
$$ (3<p \leq p' < \infty \, ,\,  3 < q\leq q'<\infty)$$
are all critical with respect to the Navier-Stokes scaling in that
$\|u_{0,\l}\|_X \equiv \|u_0\|_X$ for all $\l >0$,
where~$u_{0,\lambda}(x):= \lambda u_0(\lambda x)$ is the initial datum
which evolves as~$u_{\lambda}(t,x):= \lambda u(\lambda^2t, \lambda
x)$, as long as~$u_0(x)$ is the initial datum for the solution
$u(x,t)$.  While the larger spaces $\dot B^{-1+ \frac 3p}_{p,\infty}$,
$\mathrm{BMO}^{-1}$ and~$\dot B^{-1}_{\infty,\infty}$ are also
critical spaces and global wellposedness is known for the first two
for small enough initial data in those spaces thanks
to~\cite{Cannone,Planchon,KT}  (but only for finite~$p$ in the Besov
case, see~\cite{bourgainpavlovic}), the ones in the chain above
guarantee the existence of local-in-time solutions for any initial
datum.  Specifically, there exist corresponding ``adapted path" spaces
$X_T = X_T (\R^3 \times (0,T))$ such that for any $u_0 \in X$, there
exists $T>0$ and a unique ``strong" (or sometimes denoted ``mild") solution~$u$ belonging to~$ X_T $ to the corresponding Duhamel-type integral equation
\begin{equation}\label{NSduhamel}
\begin{aligned}
u(   t) &= \etl u_0+ \int_0^t e^{(t-s)\D} \mathbb{P} \nabla \cdot (u( s) \otimes u(   s))\, ds \\
&= \etl u_0 + B(u,u) (t) \, ,
\end{aligned}
\end{equation}
where
$$
(f\otimes g)_{j,k}:=f_jg_k \, , \quad [\nabla \cdot (f\otimes g)]_j:=\sum_{k=1}^3 \partial_k(f_jg_k) \,   \,  \mbox{and}  \,   \,
\mathbb{P}f:=f + {\nabla (-\Delta)^{-1}(\nabla \cdot f)} \, ,
$$
 which results from applying the projection onto
divergence-free vector fields operator $\mathbb{P}$  to (NS) and solving the resulting
nonlinear heat equation.  Moreover, $X_T$ is such that any $u\in X_T$ satisfying (\ref{NSduhamel}) belongs to $\mathcal{C}([0,T];X)$. Setting
\begin{equation}\label{deftstargen}
T^*_{X_T}(u_0):= \sup \{T>0 \, | \, \exists ! \, u:=NS(u_0) \in X_T \, \textrm{solving (\ref{NSduhamel})} \}
\end{equation}
to be the ``blow-up time" (if it is finite, or ``maximal time of existence") of the solution evolving from $u_0\in X$, we are interested in the following question:
\\\\
{\bf Question:}\label{thequestion}
$$\textrm{Does} \, \, \sup_{0<t<T^*_{X_T}(u_0)}\|u(\cdot , t)\|_X < \infty \quad \textrm{imply that} \quad T^*_{X_T}(u_0) = +\infty\, ?$$
Put another way, must the spatial $X$-norm of a solution become unbounded (``blow up") near a finite-time singularity?
\\\\
In the important work \cite{ess} of Escauriaza-Seregin-Sverak, it was established that for $X=L^3(\rt)$, the answer is yes (in the setting of Leray-Hopf weak solutions of (NS))\footnote{In the $L^3(\rt)$ setting of \cite{ess}, it was recently shown in \cite{seregin1} that moreover (for Leray-Hopf weak solutions) one can replace $\limsup_{t\to T^*}\|u(t)\|_{L^3} = \infty$ by $\lim_{t \to T^*}\|u(t)\|_{L^3} = \infty$ for a singular time $T^*<\infty$.}.
This extended a result in the foundational work of Leray \cite{leray} regarding the blow-up of $L^p(\rt)$ norms at a singularity with $p$ strictly greater than $3$, and of the ``Ladyzhenskaya-Prodi-Serrin" type mixed norms $L^s_t(L^p_x)$, $\frac 2s + \frac 3p =1$, $p>3$ (which follows from Leray's result), establishing a difficult ``endpoint" case of those results (as well as generalizing the result \cite{nrs} ruling out self-similar singular Leray-Hopf solutions which had been conjectured to exist in \cite{leray}).
\\\\
In our previous paper \cite{GKP}, based on the work \cite{kk} for $X=\dot H^{\frac 12}(\rt)$, we gave an alternative proof of this result in the setting of strong solutions using the method of ``critical elements" of C. Kenig and F. Merle, and in this work we extend the method in \cite{GKP} to give a positive answer to the above question for $X = \dot B^{-1+\frac 3p}_{p,q}(\rt)$
for all $3<p,q<\infty$ (see Theorem \ref{mainthm} below). In such
functional settings, our argument appears in a natural way, building
upon the local Cauchy theory, whereas extending \cite{ess} directly is in no
way straightforward.  An
important part of the proof here draws upon the intermediate result  \cite{chemplan} giving a positive answer for the same spaces in a certain range of values of $q<3$, and with an additional regularity assumption on the data.
\\\\
After completion of the present work, we learned of the very recent
work \cite{phuc}, which extends \cite{ess} to $X=L^{3,q}(\rt)$, the
Lorentz space with $3<q<+\infty$, in the context of Leray-Hopf weak solutions. In view of the embedding
$L^{3,q}(\rt)\hookrightarrow  \dot B^{-1+\frac 3p}_{p,q}(\rt)$, our
approach relies on a weaker a priori bound, but the setting and
the notion of solutions used in both works are not directly comparable.

\subsection{Besov spaces, local existence and statement of main result}
Let us first recall the definition of  Besov spaces, in dimension~$d \geq 1$.
\begin{definition}
\label{d1}
Let $\phi$ be a function in $\mathcal{S} (\R^d)$
such that $\widehat\phi(\xi) =
1$ for $|\xi|\leq 1$ and~$\widehat\phi(\xi)= 0$ for~$|\xi|>2$, and define
$
\phi_{j}(x):= 2^{dj}\phi(2^{j}x).$
Then the frequency localization operators are defined by
$$
S_{j} := \phi_{j}\ast\cdot \, , \quad
\Delta_{j} := S_{j+1} - S_{j} \, .
$$
Let $f$ be in $\mathcal{S}' (\R^d)$. We say $f$ belongs to~$\dot
B^{s}_{p,q}=\dot
B^{s}_{p,q}(\R^d)$ if
\begin{itemize}
\item[(i)] the partial sum $\displaystyle \sum^{m}_{-m} \Delta_{j} f $ converges to $f$
  as a tempered distribution if~$s < d/p$ and after taking
  the quotient with polynomials if not, and
\item[(ii)] $\displaystyle{\|f\|_{\dot B^s_{p,q}}:= \big\| 2^{js}\| \Delta_{j}  f \|_{L^{p}_x} \big\|_{\ell^q_j} <\infty}$.
\end{itemize}
\end{definition}
\noindent Note that there is an   equivalent formulation of Besov spaces by the heat flow:  defining the operator~$K_b(\tau):=(\tau \partial_\tau)^b
e^{\tau\Delta}$, then, for $b=1$, $s<2$,
\begin{equation}\label{besequiv} \|f\|_{\dot B^{s}_{p,q}} \sim
\left\| \|\tau^{-s/2} K_1(\tau) f\|_{L^p}\right\|_{L^q(\mathbb{R}^+,
\frac{d\tau}{\tau})}\,,
\end{equation}
and when $s<0$, the previous equivalence holds with $b=0$; in the local
well-posedness theory for (NS) with $L^3(\R^3)$ data, this equivalence
provides a natural link between heat decay and Besov spaces (see e.g. \cite{Planchon}).
\\\\
We shall also need a slight modification of those spaces, introduced in~\cite{cheminlerner}, taking into
account the time variable.
\begin{definition}
  \label{raah11}
  Let $u(\cdot,t)\in \dot B^s_{p,q}$ for a.e. $t\in (t_{1},t_{2})$ and let $\Delta_j$ be a
  frequency localization with respect to the $x$ variable (see Definition~\ref{d1}). We shall say
  that $u$ belongs to~${\tl}^{\rho}([ t_{1},t_{2} ] ;\dot{B}^{s}_{p,q})$   if
$$
\|u\|_{{\tl}^{\rho}([ t_{1},t_{2} ] ;\dot{B}^{s}_{p,q})} :=  \|2^{js}\|\Delta_j u\|_{L^\rho([ t_{1},t_{2} ];L^{p}_x)} \|_{\ell^q} <\infty \, .
$$
\end{definition}
\noindent
Note that for $1\leq \rho_1 \leq q\leq \rho_2 \leq \infty$, by Minkowski's and H\"older's inequalities (and Fubini's theorem when $\rho_1 = \rho_2 = q$) we have
\begin{equation}\label{minkowski}
{L}^{\rho_1}([ t_{1},t_{2} ] ;\dot{B}^{s}_{p,q}) \hookrightarrow
{\tl}^{\rho_1}([ t_{1},t_{2} ] ;\dot{B}^{s}_{p,q}) \hookrightarrow
{\tl}^{\rho_2}([ t_{1},t_{2} ] ;\dot{B}^{s}_{p,q}) \hookrightarrow
{L}^{\rho_2}([ t_{1},t_{2} ] ;\dot{B}^{s}_{p,q}) \, .
\end{equation}
Let us also introduce the following notations: we define~$s_p := -1 + \tfrac3p$ and
\begin{equation}\label{scrldef}
\begin{aligned}
 {\mathscr L}^{a:b}_{p,q}(t_1,t_2)&:= \tl^a([t_1,t_2];{\dot B^{s_{p}+ \frac 2 a}_{p,q}}) \cap
\tl^{b} ([t_1,t_2];\dot B^{s_{p}+ \frac 2 b}_{p,q}) = \bigcap_{a\leq r \leq b} \tl^{r} ([t_1,t_2];\dot B^{s_{p}+ \frac 2 r}_{p,q})\, ,\\
{\mathscr L}^{a:b}_{p}&:=  {\mathscr L}^{a:b}_{p,p}\, , \quad
{\mathscr L}^{a}_{p,q} :=  {\mathscr L}^{a:a}_{p,q}\, , \quad {\mathscr L}^{a}_{p} :=  {\mathscr L}^{a:a}_{p,p}
\, , \quad {\mathscr L}^{a:b}_{p,q}(T):={\mathscr L}^{a:b}_{p,q}(0,T)
\\
& \textrm{and} \quad
\mathscr{L}^{a:b}_{p,q}[T<T^*]:=\bigcap_{0<T<T^*}\mathscr{L}^{a:b}_{p,q}(T)
\, .
\end{aligned}
\end{equation}

\begin{remark}
Notice that the spaces~$ {\mathscr L}^{a:b}_{p,q}$ are natural in this
context since the norm in~$ {\mathscr L}^{a:b}_{p,q}(\infty)$ is
invariant through the scaling transformation~$u \mapsto
u_\lambda$. We also carefully point out that according to our notations, $u\in \mathscr{L}^{a:b}_{p,q}[T<T^*]$
merely means that $u\in \mathscr{L}^{a:b}_{p,q}(T)$ for each fixed $T<T^*$ and does not imply that $u\in \mathscr{L}^{a:b}_{p,q}(T^*)$ (the notation does not imply any uniform control as $T\nearrow T^*$).
\end{remark}
\noindent
For the convenience of the reader, we have collected the standard estimates relevant to Navier-Stokes (heat estimates, paraproduct estimates and embeddings via Bernstein's inequalities) in these spaces in Appendix \ref{paraproductsapp}.
\\\\
Let us now recall more precisely the main results on the Cauchy problem for (NS) in the setting of Besov spaces. For any divergence-free initial datum~$u_0$ in~$X:= \dot{B}^{s_p }_{p,q}$, with
$3<p<\infty$ and~$1\leq q < \infty$,
it is known (see \cite{Cannone} for $3<p\leq 6$ and~\cite{Planchon} for all
$p<+\infty$) that there is a unique solution  to (\ref{NSduhamel}),
which we shall denote by~${\rm NS}(u_0)$, which belongs to~$X_T:=
{\mathscr L}^{1:\infty}_{p,q}(T)$ for some $T>0$.  Moreover for a
fixed $p$ and $q$,
\begin{equation}\label{timecontinuous}
\left .
\begin{array}{c}
u_0\in \bespq,\, T>0,\\
u_1, u_2 \, \textrm{satisfy}\, (\ref{NSduhamel})\, \textrm{in}\, \cl^{1:\infty}_{p,q}(T)
\end{array}
\right\}
 \, \, \Longrightarrow \, \, u_1 = u_2 \in \mathcal{C}([0,T];\bespq)\cap \mathcal{C}^\infty(\rt \times (0,T])\, .
\end{equation}
Actually uniqueness holds in the class ~${\mathscr L}^{2+\e}_{p,q}(T)$ for any given $\e<6/(p-3)$,
but we shall not be using that fact here.
We also recall that there is a positive constant~$c_0$  such that if
the initial datum $u_0$ satisfies~$\|u_0\|_{ \dot{B}^{s_p }_{p,q}} \leq
c_0$, then~$NS(u_0)$ is a global solution and belongs
to~${\mathscr L}^{1:\infty}_{p,q}(\infty)$. Moreover it is known
(see~\cite{gip3} and \cite{adt} for the corresponding endpoint result
with data in $\mathrm{BMO}^{-1}$) that any   solution belonging to~${\mathscr L}^{1:\infty}_{p,q}[T <\infty]$, with notation~(\ref{scrldef}), actually belongs to~${\mathscr L}^{1:\infty}_{p,q}(\infty)$ and satisfies
\begin{equation}\label{limTinfty}
\lim_{t \to \infty} \|{\rm NS}(u_0)(t)\|_{ \dot{B}^{s_p }_{p,q}} = 0 \, .
\end{equation}
Note that by definition (\ref{deftstargen}) with $X_T={\mathscr L}^{1:\infty}_{p,q}(T)$ and in view of (\ref{limTinfty}), the maximal existence time $T^*= T^*_{{\mathscr L}^{1:\infty}_{p,q}(T)}(u_0)$
satisfies
\begin{equation}\label{deftstar}
T^* < \infty \, \, \, \iff \, \, \, \exists \rho \in [1,\infty) \, , \quad \lim_{t \to T^*} \|{\rm NS}(u_0)\|_{ {\tl}^{\rho}([ 0,t] ;\dot{B}^{s_p + \frac2\rho}_{p,q})} = \infty \, .
\end{equation}
Moreover, it is well-known (due to the embeddings (\ref{embed}) and ``propagation of regularity" results, cf., e.g., \cite{gip3})
that
\begin{equation}\label{tstarindep}
T^*_{ {\mathscr L}^{1:\infty}_{p,q}(T)}(u_{0})\ \  \textrm{is independent of}\ p \ \textrm{and} \ q \ \textrm{for any}\ p,q \in (3,\infty).
\end{equation}
Therefore, we will denote it by $T^{*}(u_{0})$ or just $T^{*}$ when
there is no ambiguity over the data.
\\\\
Our aim is to prove that in fact $$
T^* < \infty \, \, \, \Longrightarrow \, \, \, \limsup_{t \to T^* } \|{\rm NS}(u_0)(t)\|_{ \dot{B}^{s_p }_{p,q}} = \infty\,.
$$
(The converse already follows from (\ref{limTinfty}).)  More precisely, in this paper we prove the following theorem, which gives an affirmative answer to the   question raised on page~\pageref{thequestion} in the Besov space setting.
\begin{thm}\label{mainthm}
Let~$p, q \in (3,\infty)$ be given, and
consider a divergence free vector field~$u_0 $ in~$ \dot B^{s_p}_{p,q}$. Let~$u=NS(u_0) \in  {\mathscr L}^{1:\infty}_{p,q}[T<T^*]$
be the unique strong Navier-Stokes solution of (\ref{NSduhamel}) with maximal time of existence  $T^*$.  If $T^*<\infty$, then
$$
\limsup_{t \to T^*} \|u(t)\|_{ \dot B^{s_p}_{p,q}} = \infty \, .
$$
\end{thm}
\noindent
The rest of this article is structured as follows.  In Section \ref{prooftheorem},  we outline the proof of Theorem~\ref{mainthm}, leaving the proofs of the main supporting results to the subsequent sections.   The strategy of the proof,  following~\cite{kk} and~\cite{GKP} (based on the strategy of~\cite{KMNLS}-\cite{KMNLS3}), is by contradiction: assuming the conclusion of Theorem~\ref{mainthm} fails, we construct
 a ``critical element", namely a solution blowing up in finite time with minimal~$L^\infty(0,T^*; \dot B^{s_p}_{p,p})$
norm (in view of (\ref{tstarindep}) and (\ref{embed}), it is no loss of generality to set $q:=p$, which reduces some technical difficulties).
 The key tool in doing so is the ``profile decomposition" result (Theorem~\ref{thm:profa}, proved in
 Section~\ref{profiledecompo}) for solutions to (NS) associated with bounded data in~$\dot B^{s_p}_{p,p}$.  In turn we prove, again using  Theorem~\ref{thm:profa}, that such a critical element must vanish, in~${\mathcal S}'$, at blow up time, and we reach a contradiction via a backwards uniqueness argument.
\\\\
The difficulty compared with the previous references is the very low (negative) regularity of the space~$\bpp$ (for very large $p$) in which we assume control of the solution at blow-up time; in order to implement the above strategy we therefore rely on some improved regularity bounds on strong Navier-Stokes
solutions  using several iteration procedures: these are to be found
in Sections~\ref{iterationsection} and \ref{invertibilitysection}, and
the provided improved bounds are valid for any bounded local in time solution. As such,
they may prove to be useful in other contexts and are of independent
interest. Finally in Appendix~\ref{perturbationtheoryNS}
 a  perturbation result for~(NS) is stated in an appropriate functional setting which provides the key estimate in Theorem \ref{thm:profa}, and in Appendix \ref{paraproductsapp} we collect the standard Besov space estimates used throughout.
\\\\
{\bf Acknowledgments. } The authors extend their warm gratitude to Kuijie Li for
pointing out incorrect arguments in a previous version of the manuscript, and whose
remarks helped to improve the overall presentation.


\section{Proof of the main theorem \label{prooftheorem}}
\noindent
\subsection{Main steps of the proof}
\noindent The proof of Theorem~\ref{mainthm} follows the methods of~\cite{kk,GKP}. Before describing the main steps, let us start by noticing
that due to  the embedding (\ref{embed}) and the   fact recalled in
the introduction that, for $u_{0}$ belonging to ${\dot
  B}^{s_{p}}_{p,q}$,  $T^*(u_0)$ is independent of  $p$ and~$q$ for
any $p,q \in (3,\infty)$,  one can prove Theorem~\ref{mainthm} in the
case when~$p=q$, and one can also choose~$p$ as large as needed: in
the following we shall assume that $p=3\cdot 2^{k}-2$, for a given integer $k \geq 2$.\medskip
\noindent Let us define
$$
A_c := \sup \Big\{A>0\, \Big|\, \sup_{t\in
[0,T^*( u_0))}\|{\rm NS}(u_0)(t)\|_{\bpp} \leq A \Longrightarrow
T^*( u_0) = \infty \, \, \forall u_0\in \bpp \Big\} \, .
$$
Note that $A_c$ is well-defined by small-data results.  Moreover,
  if $A_c $ is finite, then
  $$
A_c = \inf \Big\{\sup_{t\in [0,T^*( u_0))}\|{\rm NS}( u_0)(t)\|_{\bpp} \, \Big| \,  u_0\in \bpp\, \, \textrm{with}\, \, T^*( u_0) < \infty \,  \Big\}\, .
$$
In the case that
$A_c<\infty$ (i.e. Theorem \ref{mainthm} is false),  we introduce the (possibly empty) set of initial data generating ``critical elements'' as follows:
$$
{\mathcal D}_c:= \Big\{
u_0 \in \bpp \,   \Big|   \,  T^*( u_0) < \infty \quad   \mbox{and} \,  \sup_{t\in [0,T^*(u_0))}\|{\rm NS}(u_0)(t)\|_{\bpp} = A_c< \infty
\Big\} \, .
$$
Theorem~\ref{mainthm} is an immediate corollary of the next three statements.
\begin{prop}[Existence of a critical element]\label{propexistcrit}
If~$A_c < \infty$, then the set~${\mathcal D}_c$ is non empty.
\end{prop}
\begin{prop}[Compactness at blow-up time of   critical elements]\label{compactcrit}
If~$A_c < \infty$, then any~$u_0$ in~${\mathcal D}_c$ satisfies
$$
{\rm {\rm NS}}(u_0)(t)
\to 0\quad \mbox{in} \quad \mathcal{S}' \,, \quad  \textrm{as} \quad  t \nearrow T^*(u_0) \, .
$$
\end{prop}
\begin{prop}[Rigidity of critical elements]\label{globalcrit}
If~$u_0$ belongs to~$\bpp$ with
$$
   \sup_{t\in [0,T^*(u_0))}\|{\rm NS}(u_0)(t)\|_{\bpp}  < \infty
$$ and if~$
{\rm NS}(u_0)(t)
\to 0$  in~$ \mathcal{S}' $ as~$ t \nearrow T^*(u_0)$, then~$T^*(u_0) = \infty$.
\end{prop}
\noindent
The proofs of Propositions \ref{propexistcrit} and \ref{compactcrit}
depend primarily on the ``profile decomposition" and related
``orthogonality" results presented in Section \ref{profdecintro} below (and proved later in Section \ref{profiledecompo}).  The proof of Proposition \ref{globalcrit} using backwards uniqueness, unique continuation and ``$\e$-regularity" results relies crucially on the ``improved bounds via iteration" results presented and proved in   Sections \ref{iterationsection} and \ref{invertibilitysection}.
\\\\
In the remainder of Section \ref{prooftheorem} we outline the proofs of Proposition \ref{propexistcrit}, Proposition \ref{compactcrit} and Proposition \ref{globalcrit}, postponing the proofs of the more technical points to the subsequent sections.

\subsection{Profile decompositions}\label{profdecintro}
\quad In~\cite{GKP} a profile decomposition of solutions to the Navier-Stokes equations associated with data in~$\dot B^{s_p}_{p,p}(\R^d)$ is proved for~$d< p < 2d+3$,
 thus extending the result of~\cite{Gallagher} which only deals with the case~$p=2$. In this section we extend (with $d=3$ for simplicity) that decomposition  to the full range of $p \in (3,\infty)$: the main new ingredient   is the decomposition (proved in Section \ref{elemdecompsec}) of any solution to the Navier-Stokes equations into two parts, the first of which  involves only the heat extension of the initial data, and the second of which is  smooth (prior to blow-up time). We refer to Lemma~\ref{decompositionNS} for a precise statement.
\\\\
Before stating the main result of this section, let us recall the following definition.
        \begin{definition}\label{deforthseq}
   We say that two sequences
 $(\lambda _{j,n} ,x_{j,n})_{n \in \N} \in ((0,\infty) \times \R^d)^\N$ for~$j \in \{1,2\}$ are orthogonal, and we write~$(\lambda _{1,n} ,x_{1,n})_{n \in \N}  \perp (\lambda _{2,n} ,x_{2,n})_{n \in \N}, $ if
 \begin{equation}\label{orthseq}
\lim_{n \to
+\infty} \frac{\lambda _{1,n} }{\lambda _{2,n} } + \frac{\lambda
_{2,n} }{\lambda _{1,n} } + \frac{|x_{1,n} -
x_{2,n}|}{\lambda _{1,n} }   = +\infty \, .
\end{equation}
Similarly we say that a set of sequences $(\lambda _{j,n} ,x_{j,n}) _{n \in \N}$,  for $j\in \N$, $ j \geq 1$,
is (pairwise)  orthogonal if for all~$ j \neq j' $, $(\lambda _{j,n} ,x_{j,n})_{n \in \N}  \perp (\lambda _{j',n} ,x_{j',n})_{n \in \N}.$
\end{definition}
\noindent Next let us define, for any set of sequences $(\lambda _{j,n} ,x_{j,n}) _{n \in \N}$ {(for $j\geq 1$)},  the scaling operator
\begin{equation}\label{defLambdatilde}
  \Lambda_{j,n} U_j (x,t) := \frac{1}{\lambda _{j,n} }
U_j \Big(  \frac{x - x_{j,n}}{\lambda _{j,n} } , \frac{t}{\lambda
_{j,n} ^2}   \Big) \, \cdotp
\end{equation}
\noindent It is proved in~\cite{gk} (based on the technique of \cite{Jaffard}) that any bounded (time-independent) sequence in~$\dot
B^{s_{p}}_{p,p}(\R^d)$ may be decomposed into a sum  of rescaled functions~$ \Lambda_{j,n} \phi_j $, where the set  of sequences~ \linebreak
$(\lambda _{j,n} ,x_{j,n}) _{n \in \N}$ is orthogonal, up to a small remainder term in~$\dot B^{s_{q}}_{q,q}$, for any~$q >p$. The precise statement is as follows, and is in the spirit of the pioneering work~\cite{PG}.
\begin{thm}[\cite{gk}]\label{thm:dataprofb}
Fix $ p, q \in [1,\infty]$ such that $p<q$. Let~$(f_n)_{n \geq 1}$ be a bounded sequence in~${\dot
B^{s_{p}}_{p,p}}(\R^d)$, and let~$\phi_{1}$ be any weak limit
point of~$(f_n)$. Then, after possibly replacing~$( f_n)_n$ by a subsequence which we relabel~$(f_n)_n$, there exists a sequence of profiles~$(\phi_j)_{j  \geq 2} $ in~${\dot B^{s_{p}}_{p,p}}(\R^d)$, and a set of sequences~$(\lambda _{j,n} ,x_{j,n}) _{n  \geq 1}$ for $j\in \N$ with $(\lambda_{1,n},x_{1,n}) \equiv (1,0)$ which are orthogonal  in
the sense of Definition \ref{deforthseq}  such that, for all~$n,J\in \mathbb{N}$,
if we define $\psi_n^J$ by
\begin{equation}\label{profilesa}
f_n =
\sum_{j=1}^J \Lambda_{j,n} \phi_j  + \psi_n ^J \, ,
\end{equation}
the following properties hold:
\begin{itemize}
\item the function $\psi_n^J$ is a remainder in the sense that
\begin{equation}\label{orth2a}
\lim_{J\to\infty}
\Big(\limsup_{n\to \infty} \|\psi_n^J\|_{\dot B^{s_{q}}_{q,q}(\R^d)}
\Big) = 0\,;
\end{equation}
\item there is a norm $\| \cdot \tilde \|_{\dot B^{s_{p}}_{p,p}(\R^d)}$
which is equivalent to $\|\cdot \|_{\dot B^{s_{p}}_{p,p}(\R^d)}$ such that
\begin{equation}\label{orth2ab}
\Big\| \Big(\| \phi_j {\tilde
\|}_{\dot B^{s_{p}}_{p,p}(\R^d)} \Big)_{j=1}^\infty \Big\|_{\ell^p}
\leq \liminf_{n\to\infty} \|f_{n}{\tilde \|}_{\dot
B^{s_{p}}_{p,p}(\R^d)}
\end{equation}
and, for each integer~$J $,
\begin{equation}\label{orth2ac}
\| \psi_n^J {\tilde \|}_{\dot
B^{s_{p}}_{p,p}(\R^d)} \leq \|f_{n}{\tilde \|}_{\dot B^{s_{p}}_{p,p}(\R^d)}
+ \circ(1) \quad \textrm{as} \quad n\to\infty\,;
\end{equation}
\end{itemize}
\end{thm}
\noindent
Notice that, in particular, for any~$j\geq 2$,  either
${\displaystyle \lim_{n\to\infty} |x_{j,n}| = +\infty}$ or
${\displaystyle \lim_{n\to \infty} \lambda _{j,n} \in \{0,+\infty\}}$ due to the orthogonality of the scales/cores with $(\lambda_{1,n},x_{1,n}) \equiv (1,0)$, and also that
\begin{equation}\label{orth3a} \Big\| \Big(\| \phi_j {
\|}_{\dot B^{s_{p}}_{p,p}(\R^d)} \Big)_{j=1}^\infty \Big\|_{\ell^p}
\lesssim \liminf_{n\to\infty} \|f_{n}{  \|}_{\dot
B^{s_{p}}_{p,p}(\R^d)}\, .
\end{equation}
\noindent
(See Remark \ref{improvedpdorth} below for an improvement of (\ref{orth3a}); in particular, one may take the constant equal to one.)  In Section \ref{profiledecompoa} (where, again, we set $d=3$ for simplicity), we shall prove the following result on the propagation of~(\ref{profilesa}) by the Navier-Stokes flow, which extends Theorem~3 of~\cite{GKP} to the full range of the $p$ index (where for very large values of $p$, we rely crucially on the iterations described in Section \ref{elemdecompsec}).

\pagebreak

\begin{thm}[{\rm NS} Evolution of Profile Decompositions]\label{thm:profa}
Fix~$p , q$ with $3< p   < q \leq \infty$. Let~$(u_{0,n})_{n\geq1}$ be a bounded sequence of
divergence-free vector fields in~${\dot B^{s_{p}}_{p,p}}$, and let~$\phi_{1}$ be any weak limit point of~$(u_{0,n})$.  Let $(u_{0,n})_{n\geq1}$ denote the subsequence given by applying Theorem~{\rm\ref{thm:dataprofb}} with $f_n:=u_{0,n}$, and let $(\Lambda_{j,n}\phi_j)_{j\geq 1}$ (with $\Lambda_{1,n} \equiv \mathrm{Id}$) and $\psi_n^J$ be the associated (divergence-free, due to~{\rm(\ref{orthseq})}) profiles and remainder.  Then setting~$T^*_j:=T^*(\phi_j)$ and denoting~$U_j:= NS(\phi_j) \in \mathscr{L}^{1:\infty}_p[T<T_j^*]$ and $u_n:= NS(u_{0,n})$, the following properties hold:
\begin{itemize}
\item there is a finite (possibly empty) subset~$I $ of
$\mathbb{N}  $ such that
$$
T_j^* < \infty \, \, \forall j \in I \, ,\qquad \mbox{and} \qquad  U_j \in
  {\mathscr L}^{1:\infty}_{p}({\infty}) \, \, \forall j
\in \mathbb{N}  \setminus I\, .
$$
For all~$j \in I$ fix any~$T_j< T_j^*$ and define~$\displaystyle \tau_n := \min_{j \in I}\lambda _{j,n} ^2
T_j $ if $I$ is nonempty and $\displaystyle \tau_n := \infty$ otherwise. Then we have  $$
\sup_n \|u_n\|_{  {\mathscr L}^{1:\infty}_{p} ({\tau_n})}<\infty \, ;
$$

\item  setting $w_n^J:=\etl \psi_n^J$,  there exists some~$J_0 \in \N$  and $N(J)\in \N$ for each $J>J_0$ such that $r_n^J$ given by
\begin{equation}\label{evolprof} u_n (x,t) = U_1(x,t) + \sum_{j =
2}^{J}  \Lambda_{j,n} U_j (x,t)+ w_n^J (x,t) + r_n^J
(x,t)
\end{equation}
is well-defined for $J>J_0$, $n>N(J)$, $t< \tau_n
$ and $x\in \rt$, and moreover $w_n^J$ and $r_n^J$ are small remainders in the sense that
\begin{equation}\label{remainders} \lim_{J\rightarrow \infty}
\Big(\limsup_{n \rightarrow \infty}
\|w_n^J\|_{ {\mathscr L}^{1:\infty}_{q}(\infty)}\Big) = \lim_{J\rightarrow \infty}
\Big(\limsup_{n \rightarrow \infty}
\|r_n^J\|_{ {\mathscr L}^{2:\infty}_{q}(\tau_n)}\Big) = 0 \, .
\end{equation}
\end{itemize}
 \end{thm}

\begin{remark}\label{russverextend}
As in \cite[Theorem 9]{GKP}, Theorem~{\rm\ref{thm:profa}} above automatically implies the existence (and relevant compactness) of ``minimal blow-up initial data" in the spaces $\bespq$ for all $p,q \in (3,\infty)$.  This extends the range in \cite{GKP}, which was restricted to $p,q < 9$.  Moreover, due to Remark {\rm\ref{improvedpdorth}} below which improves  the constant to one in inequality  {\rm(\ref{orth3a})}, the results are true in the original Besov norm given in Definition  {\rm\ref{d1}}(and not just in the equivalent wavelet norm used in \cite{gk}).  All of these results generalize the original result in \cite{sverakrusin} which treated the $\dot H^{\frac 12}(\rt)$ case.
\end{remark}
\noindent
We also have the following important orthogonality result, which is the analogue of Claim 3.3 of~\cite{GKP} and will be proved in Section \ref{proofclaim33}. To state the result, note first that in case an application of Theorem~\ref{thm:profa} yields a non-empty blow-up set $I$, thanks to~(\ref{orth3a}) we also know that there is some $J^* \in \N$ such that after re-ordering the profiles
\begin{equation} \label{defJ0}
T_j^* < \infty \iff 1\leq j \leq J^* \, .
\end{equation}
Then we can again re-order those first~$J^*$ profiles, thanks to the orthogonality (\ref{orthseq}) of the scales~$\lambda_{j,n}$, so that for
$n_0 = n_0(J^*)$ sufficiently large, we have
\begin{equation}\label{reorder}
 \forall n\geq n_0 \, ,\qquad 1 \leq  j\leq j' \leq J^* \quad \Longrightarrow \quad \lambda
_{j,n} ^2 T_j^* \: \: \leq \: \: \lambda_{j',n} ^2 T_{j'}^* \, .
\end{equation}
\begin{prop}\label{claim33}
Let~$(u_{0,n})_n$ be a sequence of divergence-free data which are bounded in~$\bpp$ and for which the set $I$ of blow-up profile indices resulting from an application of Theorem {\rm\ref{thm:profa}} is non-empty.  After re-ordering the profiles in the profile decomposition of~$u_n:={\rm NS}(u_{0,n})$ so that
{\rm(\ref{defJ0})} and~{\rm(\ref{reorder})} hold for some $J^* \in \N$, setting $t_n:=
\lambda _{1,n}^2s$ for $s\in [0,T^*_1)$  one
has (after possibly passing to a subsequence in $n$)
$$
\|u_n(t_n)\|_{\bpp}^p = \|\big(  \Lambda_{1,n}U_1 \big) (t_n)\|_{\bpp}^p  +
\|u_n(t_n) -\big(  \Lambda_{1,n}U_1 \big) (t_n)\|_{\bpp}^p +\e(n,s) \, ,
$$
where $\e(n,s) \to 0$ as~$n \to \infty$ for each fixed $s\in [0,T^*_1)$.
\end{prop}

\begin{remark}\label{improvedpdorth}
Note that the proof of Proposition~{\rm\ref{claim33}} (which does not use any special property of the first profile in particular, unlike our proof of the analogous result in \cite{GKP}) actually shows that we may improve {\rm(\ref{orth3a})} to a true orthogonality of the original profile decomposition ($s=0$) of the form
\begin{equation}\label{improvedortheq}
\|f_n\|_{\bpp}^p = \sum_{j=1}^J\|\phi_j\|_{\bpp}^p  +
\|\psi_n^J\|_{\bpp}^p +\e(n,J) \, , \quad \lim_{J\to\infty}\limsup_{n\to\infty} \e(n,J) = 0
\end{equation}
(analogous to the orthogonality proved in \cite{PG}, but lacking in \cite{Jaffard,gk}), thus improving the original result in \cite{gk}.  The above orthogonality on the flows could similarly be improved to include the other profile flows and remainders (which would extend~{\rm(\ref{improvedortheq})} to $s>0$), but it is sufficient as stated for our purposes.
\end{remark}


\subsection{[Step 1]  Existence of a critical element}
To prove Proposition~\ref{propexistcrit}, we explicitly construct an element of~${\mathcal D}_c$: this turns out to be a profile of the initial data of a minimizing sequence of~$A_c$. So let us consider a sequence~$u_{0,n}$, bounded in the space~$\bpp$, such that its life span satisfies~$T^*(u_{0,n}) < \infty$ for each $n\in \N$ and such that~$ A_n:= \sup_{t\in [0,T^*(u_{0,n}))}\|{\rm NS}(u_{0,n})(t)\|_{\bpp}  $ satisfies
$$
A_c \leq A_n \quad \mbox{and} \quad A_n \to A_c \,, \quad n \to \infty \, .
$$
Applying Theorem~\ref{thm:profa} above to $u_{0,n}$
we find that, in the notation of  Theorem~\ref{thm:profa} (and up to a subsequence extraction), for all~$t < \tau_n$,
 the solutions $u_n = {\rm NS}(u_{0,n})$ satisfy
$$
u_n(t) = \sum_{j=1}^J   \Lambda_{j,n}U_j(t) + w_n^J(t) + r_n^J(t)
$$
with~$U_j={\rm NS}(\phi_j)$ where~$(\phi_j)_{j \geq 1}$ are the
profiles of~$u_{0,n}$ according to the initial data decomposition
provided in Theorem~\ref{thm:dataprofb} (with $f_n:=u_{0,n}$), and for $q>p$ as in Theorem~\ref{thm:profa}, recall that  $$
 \lim_{J\rightarrow \infty}
\big(\limsup_{n \rightarrow \infty}
\|w_n^J\|_{ {\mathscr L}^{1:\infty}_{q}(\infty)} \big) = \lim_{J\rightarrow \infty}
\big(\limsup_{n \rightarrow \infty}
\|r_n^J\|_{ {\mathscr L}^{2:\infty}_{q}( \tau_n )} \big) = 0 \, .
$$
Defining $T^*_j:=T^*(\phi_j)$ to be the life span of~$U_j=\mathrm{NS}(\phi_j)$,
Theorem~\ref{thm:profa}  also ensures that there is~$j_0\in \N$ such that~$T^*_{j_0}< \infty$ (if not we would have~$\tau_n\equiv \infty$ and
hence~$T^*(u_{0,n}) \equiv \infty$, contrary to our assumption), and hence we may re-order the profiles so that with the new ordering (\ref{defJ0}) and (\ref{reorder}) hold for some $J^* \in \N$. Notice that in particular~$T^*_1 < \infty$, hence by definition of~$A_c$ we know that
\begin{equation}\label{largerthanAc}
\sup_{s\in [0,T^*_1)} \|U_1(s)\|_{\bpp}
\geq A_c\, .
\end{equation}
\noindent Then  Proposition \ref{claim33} above implies that  for any~$s\in (0,T^*_1)$, setting~$t_n:=
\lambda _{1,n}^2s$,  $$
\begin{aligned}
A_n:= \sup_{t\in [0,T^*(u_{0,n}))} \|u_n(t)\|_{\bpp} & \geq
\|u_n(t_n)\|_{\bpp} \\
& \geq \|\big(  \Lambda_{1,n}U_1 \big) (t_n)\|_{\bpp} + \e(n,s) \\
& =  \|U_1(s)\|_{\bpp} + \e(n,s)
\end{aligned}
 $$
which with~(\ref{largerthanAc}) and the fact that~$A_n \to A_c$ as~$n \to \infty$ implies that~$\phi_1 $ belongs to~$ {\mathcal D_c}$. Proposition~\ref{propexistcrit} is proved. \hfill $\Box$
\\

\subsection{[Step 2] Compactness at blow-up time of  critical elements}
To prove Proposition~\ref{compactcrit} we choose~$u_{0,c} \in {\mathcal D}_c$ (such an element exists if $A_c <\infty$ thanks to Proposition~\ref{propexistcrit}) and we pick a sequence of times~$s_n$ such that~$s_n \nearrow T^*(u_{0,c})$. We then define the sequence~$u_{0,n}:=u_c(s_n)$, where~$u_c:={\rm NS}(u_{0,c})$, which is bounded and to which we apply Theorem~\ref{thm:profa} (and pass to the subsequence given there). As in the proof of Proposition~\ref{propexistcrit} above we may re-arrange the first~$J^*$ terms of the profile decomposition so that~{\rm(\ref{defJ0})} and~{\rm(\ref{reorder})} hold and we have clearly
$$
 \lambda _{1,n}^2 T^*_1 \leq
T^*(u_{0,n}) = T^*(u_{0,c}) - s_n
$$
for large $n$, and hence
\begin{equation}\label{lambda0nto0}
\lambda _{1,n} \to 0 \,  , \quad n\to\infty \, .
\end{equation}
Let us denote by~$j_0$ the (unique)
index, after this re-numbering, satisfying~$\lambda_{j_0,n} \equiv 1$ and~$x_{j_0,n} \equiv0$, so that~$\phi_{j_0}$ is the weak limit of~$u_{0,n}$. Note that due to (\ref{lambda0nto0}), $j_0 \neq 1$.  To prove the proposition we need to show that~$\phi_{j_0} \equiv0$.
\\\\
As in the the proof of Proposition~\ref{propexistcrit} again, Proposition~\ref{claim33}
 implies that~$NS(\phi_1)$ is a
critical element ($\phi_1 \in \mathcal{D}_c$) since we have
$$A_n:= \sup_{t\in [0,T^*(u_{0,n}))} \|{\rm NS}(u_{0,n})(t)\|_{\bpp} = \sup_{t\in [s_n,T^*(u_{0,c}))} \|u_c(t)\|_{\bpp} \equiv A_c $$
for all $n$, due to the definition of $A_c$ and the fact that
$T^*(u_{0,c}) < \infty$.  Now let~$\varepsilon>0$ be fixed, and
choose~$s\in (0,T^*(u_{0,c}))$ such that, writing~$U_1:={\rm
  NS}(\phi_1)$,
$$
A_c^p -
\|U_1(s)\|_{\bpp}^p < \varepsilon /2\, ,
$$
which is possible thanks to the time-continuity in (\ref{timecontinuous}) of solutions to~(NS) in~$\bpp$. Proposition~\ref{claim33} and Sobolev embeddings (since $q>p$, cf. (\ref{bernst}))
 then imply that, defining~$u_n:= {\rm NS}(u_{0,n})$ and~$t_n:=\lambda
 _{1,n}^2s$,
 \begin{align*}
   A_c^p  & \geq \|u_n(t_n)\|_\bpp^p \\
          & \geq \|U_1(s)\|_\bpp^p+ \big \|\sum_{j=2}^{J} ( \Lambda_{j,n}U_j)(t_n)+ w_n^J(t_n) + r_n^J(t_n)\big \|_{{\dot B^{s_p}_{p,p}}}^p +  \epsilon (n,s)\\
          & \geq \|U_1(s)\|_{\bespp}^p
+ C\|\sum_{j=2}^{J}
\Lambda_{j,n}U_j(t_n) + w_n^J(t_n) + r_n^J(t_n) \|_{{\dot B^{s_q}_{q,q}}}^p + \epsilon (n,s)
 \end{align*}
\\\\
 where~$\epsilon (n,s) \to 0$ as $n\to \infty $.
\\\\
Choosing~$J$ large enough so that
 $$
 C\|w_n^J(t_n) + r_n^J(t_n) \|_{\dot B^{s_q}_{q,q}}^p \leq \varepsilon /2\, ,
  $$
for sufficiently large $n$,  we find that
\begin{equation}\label{preorthogonalitybesovprofiles}
\big\|\sum_{j=2}^{J}( \Lambda_{j,n}U_j)(t_n)\big\|_{{\dot B^{s_q}_{q,q}}}^q \lesssim   \varepsilon -   \epsilon (n,s) \, .
\end{equation}
But orthogonality arguments (see the proof of~\cite[Lemma 3.6]{GKP}) show that
\begin{equation}\label{orthogonalitybesovprofiles}
\sum_{j=2}^{J}  \| ( \Lambda_{j,n}U_j)(t_n)\|_{{\dot B^{s_q}_{q,q}}}^q
= \|\sum_{j=2}^{J} ( \Lambda_{j,n}U_j)(t_n)\|_{{\dot B^{s_q}_{q,q}}}^q+ \varepsilon(J,n)
\end{equation}
where for each~$J$, $\varepsilon(J,n)
 \to 0$ when~$n \to \infty$. In particular for~$j = j_0$, (\ref{preorthogonalitybesovprofiles}) and~(\ref{orthogonalitybesovprofiles}) together {(along with (\ref{timecontinuous}))} imply
 $$
 \|\phi_{j_0}  \|_{{\dot B^{s_q}_{q,q}}} = \|U_{j_0}(0) \|_{{\dot B^{s_q}_{q,q}}} \lesssim \varepsilon
 $$
since~$t_n \to 0$ as~$n \to \infty$, and hence~$\phi_{j_0} \equiv0$ which proves Proposition~\ref{compactcrit}. \hfill $\Box$
\\

\subsection{[Step 3]  Rigidity of critical elements}
The proof of Proposition~\ref{globalcrit} (which functions here as a ``rigidity theorem", in the ``concentration-compactness" proof of Theorem \ref{mainthm},  cf. e.g. \cite{KMNLS}) is based on a backwards uniqueness argument similar to that in \cite{ess} (see also~\cite{GKP,kk}).
However in order to implement this argument we need
to recover some positive regularity on the solution near blow up time. This is the purpose of the next statement, proved in
 Section \ref{iterationsection} and Section \ref{invertibilitysection} below.
\begin{prop}[Positive regularity at blow-up]\label{lemmaimprovedb}
Fix $p=3\cdot 2^{k}-2$, with an integer $k\geq 2$.   For $u_0 $ belonging to~$ \dot B^{s_p}_{p,p}$, set $T^*:=T^*(u_0)$ and define the associated solution~$u:=NS(u_0) $ in~$ \mathscr{L}^{1:\infty}_{p}[T<T^*]$.  If $T^*<\infty$ and~$u\in L^\infty(0,T^*;\dot B^{s_p}_{p,p})$, then there exist $ v, w \in \mathscr{L}^{1:\infty}_{p}[T<T^*]$ such that
$$u =  v +  w \quad \textrm{in} \, \,   \,  \mathscr{L}^{1:\infty}_{p}[T<T^*]$$
and such that moreover, for some   $\varepsilon \in (0,T^*)$,
$$
v \in L^\infty(T^*-\varepsilon,T^*;L^p(\rt)) \quad \textrm{and} \quad
w \in L^3(0,T^*;\ltrt)\, .
$$
\end{prop}
\noindent Let us apply Proposition~\ref{lemmaimprovedb}, to $u=NS(u_0)$ and $T^* = T^*(u_0)$ as in the assumptions of Proposition~\ref{globalcrit}. If $T^* <\infty$ then we can write $u$ as above, for some such $v$, $w$ and $\varepsilon $.
As $T^*<\infty$, we moreover have
$$v\in L^p(\rt \times (T^* - \varepsilon,T^*))\, .$$
Fix any $R>0$ and set
$$Q_{\varepsilon,R}(x):= \{(y,t)\in \rt \times \R \, | \, |y-x| < R\, , \, t\in (T^* - \varepsilon,T^*)\} \, .$$
As $p>3$, for fixed $\varepsilon, R >0$ we have
$$\|u\|_{L^3(Q_{\varepsilon,R}(x))} \lesssim \|v\|_{L^p(Q_{\varepsilon,R}(x))} +\|w\|_{L^3(Q_{\varepsilon,R}(x))} \longrightarrow 0 \quad \textrm{as} \quad |x| \to \infty \, .  $$
This is the key ``smallness" required in the ``$\varepsilon$-regularity" theory for ``suitable weak solutions".  That theory requires similar estimates for the pressure.  Since we consider ``mild" solutions (solutions to~(\ref{NSduhamel})), $u$ actually satisfies (NS) with pressure $\pi$ given by
$$\pi = \mathcal{R} \otimes \mathcal{R}:u \otimes u = \sum_{j,k=1}^3 \mathcal{R}_j \mathcal{R}_k (u_j u_k)\, ,$$
where $\mathcal{R}:=(\mathcal{R}_1,\mathcal{R}_2,\mathcal{R}_3)$ and where $\mathcal{R}_j$ is the $j$-th Riesz transform given by the Fourier multiplier $i\xi_j / |\xi|$.  Hence we may write
$$\pi = \mathcal{R} \otimes \mathcal{R}:(v+w) \otimes (v+w) = \mathcal{R} \otimes \mathcal{R}:[v \otimes v+w \otimes v+v \otimes w+w \otimes w]$$
and hence the standard Calder\'on-Zygmund estimates imply that
$$\pi \in L^{\frac p2}(\rt \times (T^* - \varepsilon,T^*)) + L^{\frac{p_1}2}(\rt \times (T^* - \varepsilon,T^*)) + L^{\frac 32}(\rt \times (T^* - \varepsilon,T^*))$$
for some $p_1 > 3$.  Hence in a similar way as above we have
$$\|\pi\|_{L^{\frac 32}(Q_{\varepsilon,R}(x))}  \longrightarrow 0 \quad \textrm{as} \quad |x| \to \infty \, .  $$
As we now have
$$(u,\pi) \in L^3_{\textrm{loc}}(\rt \times (T^* - \varepsilon,T^*)) \times L^{\frac 32}_{\textrm{loc}}(\rt \times (T^* - \varepsilon,T^*))$$
with the above spatial decay of the local norms, we can conclude as in \cite{kk} (since moreover $u$ belongs to~$ \mathcal{C}^\infty(\rt \times (0,T^*))$, cf. (\ref{timecontinuous})) that $(u,\pi)$ forms a suitable weak solution and is smooth at and near the time $T^*$ outside of some large compact set $K\subset \rt$.  Hence if $u(t) \to 0$  in~$ \mathcal{S}' $ as~$ t \nearrow T^*$, we can conclude that actually $u(x,T^*) \equiv 0$ for all $x\in K^c$, and backwards uniqueness and unique continuation (note that in applying the latter we need the smoothness inside $K$ at earlier times provided by (\ref{timecontinuous})) applied to the vorticity $\omega:= \nabla \times u$ as in \cite{ess} allow us to conclude that in fact $u(\cdot , t) \equiv 0$ for some~$t \in (0,T^*)$; we refer to~\cite{kk} for more details, including the statements of the backwards uniqueness and unique continuation results. Therefore  $T^* = \infty$ by small data results, contrary to assumption, which proves   Proposition~\ref{globalcrit}. \hfill $\Box$
\\\\
Theorem  \ref{mainthm}  is now proved.  In what follows, we shall prove all of the results stated without proof above.

 \section{Besov space profile decompositions for solutions to Navier-Stokes}\label{profiledecompo}
 \subsection{The Navier-Stokes evolution of profile decompositions: proof of Theorem \ref{thm:profa}}\label{profiledecompoa}  \quad  The proof of Theorem \ref{thm:profa} follows closely the arguments of~\cite{GKP}, up to the fact that we are considering rough initial data since~$p$ is arbitrarily large (but finite).  We first use   Theorem \ref{thm:dataprofb} to decompose the sequence of initial data, and then with the notation of  Theorem \ref{thm:profa}
  we write
 $$
 u_n:= {\rm NS}(u_{0,n}) \, , \quad U_j := {\rm NS}(\phi_j)\in  {\mathscr L}^{1:\infty}_{p}[T<T_j^*]\quad
\mbox{and} \quad w_n^J := e^{t\Delta}(\psi_n^J) \in  {\mathscr L}^{1:\infty}_{p}(\infty) \, .
 $$
In view of (\ref{orth2a}), the standard linear heat estimate (\ref{dataheatest}) implies    that
 \begin{equation}
   \label{eq:heatw}
   \lim_{J\rightarrow \infty} \Big(\limsup_{n \rightarrow \infty}
\|w_n^J\|_{  {\mathscr L}^{1:\infty}_{q}(\infty)} \Big) = 0\, .
 \end{equation}
Due to the stability property~\eqref{orth3a}, the
sequence~$\displaystyle( \phi_j)_{j \geq 1}$ goes to zero in the space~${\dot B^{s_{p}}_{p,p}}$ as~$j$ goes to
infinity.  This implies that
there is~$J_0 $ such that for all~$j > J_0$, there is a global unique
solution associated with~$\phi_j$, as $\|\phi_j\|_{\dot
  B^{s_p}_{p,p}}<c_0$ (the smallness constant of small data
theory). Hence,~$I$ will be a subset of $\{1,\dots,J_0\}$ which
proves the first part of the first statement in Theorem~\ref{thm:profa}.
Note that~\eqref{orth3a} implies, by Sobolev embeddings along with the fact that~$\ell^p \hookrightarrow \ell^q$, that
 \begin{equation}
   \label{ortho33}
\sum_{j\geq 1} \|\phi_j\|_{\dot B^{s_q}_{q,q}}^q  < \infty \, .
 \end{equation}
\noindent
By the local Cauchy theory we can solve the Navier-Stokes system with data~$u_{0,n}$ for each
integer~$n$, and produce a unique mild solution~$u_n \in {\mathscr L}^{1:\infty}_{p}[T<T^*(u_{0,n}) ]$.
Now let us define, for any~$J \geq1$
$$
r_n^J   := u_n  - \big( \sum_{j = 1}^{J}  \Lambda_{j,n}
U_j  + w_n^J  \big)\,,
$$
where we recall that $  \Lambda_{1,n} {U}_{1}
 := U_1 $. Note that the lifetime of $  \Lambda_{j,n}U_j$ is $\lambda_{j,n}^2 T_j^*$, where~$ T_j^*$ is the lifetime of~$\phi_j$. Therefore, the function~$r_n^J(x, \cdot)$ is defined a priori for~$t \in
[0,t_n)$, where
$$
t_n = \min (T^*(u_{0,n}); \tau_n)
$$
with the notation of Theorem~\ref{thm:profa}.  Our main goal is to prove that~$r_n^J$ is actually defined on~$[0,\tau_n^*)$ (at least
if~$J$ and $n$ are large enough),
which will be a consequence of
perturbation theory for the Navier-Stokes equations, recalled in Appendix~\ref{perturbationtheoryNS}. In the process, we
shall obtain the desired uniform limiting property, namely
$$
 \lim_{J\rightarrow \infty} \Big(\limsup_{n \rightarrow \infty}
\|r_n^J\|_{{\mathscr L}^{2:\infty}_{q}( \tau_n )} \Big) = 0\,.
$$
Let us write the equation satisfied by~$r_n^J$.  It turns out to be more convenient
to write that equation after a re-scaling in space-time. For
convenience and similarly to~(\ref{defJ0})-(\ref{reorder}), let us also re-order the functions~$ \Lambda_{j,n} U_j $, for~$1
\leq j \leq J_0$, in such a way that, for some $n_0 = n_0(J_0)$ sufficiently large, we have
\begin{equation}
\label{eq:increasTglobal} \forall n\geq n_0 \, ,\qquad  1 \leq j\leq j' \leq J_0 \quad \Longrightarrow \quad \lambda
_{j,n} ^2 T_j^* \: \: \leq \: \: \lambda_{j',n} ^2 T_{j'}^*
\end{equation} (some of these terms may equal infinity), where recall that~$ T_j^*$ is the maximal life span
of~$\phi_j$ (such a reordering is possible on a fixed and finite number of profiles
due to the orthogonality of scales). In particular, with this ordering we have $\tau_n = \l_{1,n}^2 T_1$.
We then define, for every integer~$J$,
$$
\begin{aligned}
\forall 1 \leq  j \leq J, \quad U_n^{j,1}  :=    \Lambda^{-1}_{1,n}   \Lambda_{j,n} U_j, \quad
R_n^{J,1} :=    \Lambda^{-1}_{1,n} r_n^J \,,
\\
W_n^{J,1} :=    \Lambda^{-1}_{1,n} w_n^J \quad \mbox{and} \quad U_n^{1} :=   \Lambda^{-1}_{1,n} u_n \,.
\end{aligned}
$$
Clearly we have
$$
 R_n^{J,1}    = U_n^{1} - \big( \sum_{j = 1}^{J}
U_n^{j,1}   + W_n^{J,1}  \big)  \, ,
$$
and~$ R_n^{J,1}$ (which  for the time being is defined for times~$t$ in~$ [0,T_n^1)$ where~$T_n^1 := \min \{T_1, \l_{1,n}^{-2}T^*(u_{0,n}) \}$) is a divergence free vector field,  solving the following system (in a Duhamel sense similar to (\ref{NSduhamel})):
\begin{equation}\label{remaindereq} \displaystyle
\left\{\begin{array}{rcl} \displaystyle \partial_t R_n^{J,1} +
\mathbb{P} ( (R_n^{J,1} \cdot \nabla) R_n^{J,1} ) -  \Delta R_n^{J,1}
+ Q( R_n^{J,1},F_n^{J,1}) & = & G_n^{J,1} \\
R^{J,1}_{n |s=0} & = & 0 \, ,
\end{array} \right.
\end{equation}
where we recall that~$\mathbb{P}:= \mbox{Id}-\nabla \Delta^{-1}(\nabla\cdot)$ is the
projection onto divergence free vector fields, and where
$$
Q(a,b):=\mathbb{P}((a\cdot \nabla) b+(b\cdot \nabla) a)
$$ for two vector
fields $a,b$. Finally, we have defined
$$
F_n^{J,1} := \sum_{j =1 }^J U_n^{j,1}  + W_n^{J,1} \, ,
$$
and
$$
  G_n^{J,1} :=-\frac{1}{2}Q ( W_n^{J,1}, W_n^{J,1}) - \frac{1}{2}\! \! \! \!  \sumetage{j \neq j' }{(j, j')
\in \{1,..,J\}^2} Q( U_n^{j,1}  , U_n^{j',1})- \sum_{j =1 }^J Q(
U_n^{j,1}  , W_n^{J,1} )   \, .
$$
  In order to use perturbative
bounds on this system,  we need a
uniform control on the drift term $F_n^{J,1}$, and smallness of the forcing term
$G^{J,1}_n$.
The results are the following.
\begin{lemma} \label{lemmadrift}
Fix~$T_1<T_1^*$. For any real number~$a$ satisfying~$1-3/q< 1/a <1$ and any integer~$N \geq 2 $ such that~$3(N-1)\leq q$, there is~$C>0$ such that defining
$V:={\mathscr L}^q_q(T_1) \cap \big({\mathscr L}^{
a:\infty}_{q}(\infty) +   {\mathscr L}^{\frac qN}_{\frac qN}(T_1)\big)$,
we have
$$
\lim_{J \to \infty} \limsup_{n \to \infty} \|F_n^{J,1}\|_{V} \leq C \, .
$$
\end{lemma}
\begin{lemma}\label{lemmasource}
Fix~$T_1<T_1^*$, and~$N \geq 3$ be an integer such that~$3(N-1)< q$.  For~$\delta\in (0,1)$, define~$1/r = N/q + (1-\delta)/2$ and let~$q'$ be the conjugate exponent to~$q$:~$1/q+1/q' = 1$.

\noindent
The source term~$G_n^{J,1}$ goes to zero for each~$J \in \N$, as~$n$ goes to infinity, in the space~$\tl^{q'}([0,T_1];\dot B^{s_q -  \frac2 {q}}_{q,q}) + \tl^{r}([0,T_1];\dot B^{s_q -1-\delta +\frac{2N}q}_{q,q})$.
\end{lemma}
\noindent
Assuming these lemmas to be true,   the end of the proof of the theorem is a direct consequence of Proposition~\ref{pdr1}. \hfill $\Box$
\\\\
So   let us prove Lemmas~\ref{lemmadrift} and~\ref{lemmasource}.
\\\\
{ \bf Proof of Lemma \ref{lemmadrift}:} \quad
This Lemma improves on~\cite[Lemma 3.5]{GKP}, thanks to Lemma~\ref{decompositionNS} below.
We first note that the uniform bound in~$  {\mathscr L}^q_q(T_1)$ is due to~\cite[Lemma 3.5]{GKP}.  Next  due to~(\ref{eq:heatw}) and scale invariance, we know that
$$
  \lim_{J\rightarrow \infty} \limsup_{n \rightarrow \infty}
\|W_n^{J,1}\|_{{\mathscr L}^{1:\infty}_{q}(\infty)}  =0 \, .
$$
\smallskip
\noindent Now let us turn to~$\displaystyle \sum_{j \leq J } U_n^{j,1}  $.  We will need the following rather elementary decomposition result for any solution to~({\rm NS}), whose proof we postpone until the next section:
  \begin{lemma}\label{decompositionNS}
Let~$u_0 \in \dot B^{s_p}_{p,p}$ be given, with~$3< p  < \infty$,
and let~$u:={\rm NS}(u_0)$ belong to~$ {\mathscr L}^{1:
\infty}_{p}(T)$ for some~$T>0$.
For any integer~$N\geq 2$ such that~$ 3(N-1) < p$, there are~$H_N \in {\mathscr L}^{1:
\infty}_{p}(\infty)$ and~$Z_N \in
 {\mathscr L}^{\frac pN}_{\frac pN}(T)
$ such that~$u = H_N + Z_N$.
\\\\
Moreover, we have
$
H_N = {\mathscr{B}}_{N-1}(u_L)$, where~$u_L:=e^{t \Delta}u_0$ and~${\mathscr{B}}_{N}$ is a  finite sum of multilinear operators of order   at most~$N$, independent of~$u_0$ and~$p$.
\\\\
Finally the following results hold.
\begin{enumerate}
\item  If~$\|u_0\|_{\dot B^{s_p}_{p,p}} \leq c_0$, the small constant
  which guarantees $T^*(u_0)=\infty$, then
 $$
 \|Z_N\|_{ {\mathscr L}^\frac pN_{\frac pN}(\infty)} \lesssim \|u_0\|_{\dot B^{s_p}_{p,p}} ^N \, .
 $$

\item For any scaling operator~$  \Lambda_{j,n}$ as in (\ref{defLambdatilde}), the commutation property holds:
$$
  \Lambda_{j,n} {\mathscr{B}}_N (\cdot) = {\mathscr{B}}_N (  \Lambda_{j,n} \cdot)  \, .
$$

\item Writing~${\mathscr{B}}_{N}$ as the finite sum of~$\ell$-linear operators~$(B_\ell)_{1 \leq \ell \leq N}$, then
  for any~$\ell \geq 2$ and for any set~$(U_j)_{1 \leq j \leq \ell}$ in~${\mathscr L}^{1:\infty}_{p}(\infty)$,
  \begin{equation}\label{estimateB1}
  \big \|  {  B}_{\ell}\big(  U_1, \dots,   U_\ell  \big)\big \| _{{\mathscr L}^{1:\infty}_{p}(\infty)} \lesssim
  \prod_{j=1}^\ell \|U_j\|_{{\mathscr L}^{1:\infty}_{p}(\infty)}  \, .
  \end{equation}
  Moreover for any~$a$ such that~$ 1-d/p <1/a <1$,  we have
 \begin{equation}\label{estimateB2}
\begin{aligned}
\big \|  {  B}_{\ell}\big(   \Lambda_{j_1,n} U_1, \dots,     \Lambda_{j_\ell,n} U_\ell  \big)\big \| _{{\mathscr L}^{a:\infty}_{p}(\infty)} \to 0 \, , \quad n \to \infty \, ,
\end{aligned}
  \end{equation}
provided there exists $k\neq k'$  with $1\leq k, k' \leq \ell$ such that~$ (\lambda_{j_k,n} , x_{j_k,n} ) \perp  (\lambda_{j_{k'},n} , x_{j_{k'},n} )$.
 \end{enumerate}

\end{lemma}
\noindent
Continuing with the proof of Lemma~{\rm\ref{lemmadrift}}, for each~$j
\leq J$ we apply the decomposition provided in
Lemma~\ref{decompositionNS}: we write, with similar notation as in the
lemma, for any integer~$N $ in~$[2, \frac q3 +1]$,
$$
U_j = H_{N,j} + Z_{N,j} \,
$$
where~$H_{N,j}$ is the sum of a finite number of multilinear operators of order at most~$N-1$, acting on the vector field~$u_{L,j}:=e^{t \Delta}\phi_j$ only, while~$Z_{N,j} $ belongs to~${\mathscr L}^\frac qN_{\frac qN}( T_j)$. It follows that $$
U_n^{j,1}  =     \Lambda^{-1}_{1,n}   \Lambda_{j,n} ( H_{N,j} + Z_{N,j}) \, .
$$
so we can write
$$
 \sum_{j \leq J } U_n^{j,1}   = \sum_{j =1 }^J       \Lambda^{-1}_{1,n}   \Lambda_{j,n} ( H_{N,j} + Z_{N,j})   \,.
$$
Let us start with the study of~$\displaystyle \sum_{j =1 }^J       \Lambda^{-1}_{1,n}   \Lambda_{j,n} H_{N,j} $. For each fixed~$N$ we can write
$$
H_{N,j} = \sum_{\ell = 1}^{N-1}{  B}_{\ell}\big( (e^{t \Delta}\phi_j) ^{\otimes \ell}\big) \, ,
$$
where  as in the statement of Lemma~\ref{decompositionNS},  $ {  B}_{\ell}\big( a^{\otimes \ell}\big)$ denotes a generic~$\ell$-linear operator applied to~$a$. Moreover thanks to Lemma~\ref{decompositionNS} (2),
$$
\begin{aligned}
   \Lambda_{j,n}  H_{N,j} & =  \sum_{\ell = 1}^{N-1}{  B}_{\ell}\big( (   \Lambda_{j,n} e^{t \Delta}\phi_j) ^{\otimes \ell}\big) \\
 & =  \sum_{\ell = 1}^{N-1}{  B}_{\ell}\big( (  e^{t \Delta}   \Lambda_{j,n}\phi_j) ^{\otimes \ell}\big)
\end{aligned}
$$
by the scaling of the heat flow, so we can write $\displaystyle
\sum_{j =1 }^J    \Lambda_{j,n}  H_{N,j} =  H_{n,N}^{(1)} + H_{n,N}^{(2)}  \, ,
$
where
$$
\begin{aligned}
H_{n,N}^{(1)} & := \sum_{\ell = 1}^{N-1}
{  B}_{\ell}\big( (  \sum_{j =1 }^J  e^{t \Delta}   \Lambda_{j,n}\phi_j) ^{\otimes \ell}\big) \quad \mbox{and} \\
H_{n,N}^{(2)} &:= -\sum_{\ell = 2}^{N-1} \sumetage{\{j_1,\dots,j_\ell\} \in \{1,\dots , J\}}{ \exists k , k' , j_k \neq j_{k'} }{  B}_{\ell}\big(   e^{t \Delta}     \Lambda_{j_1,n}  \phi_j, \dots,  e^{t \Delta}     \Lambda_{j_\ell,n}  \phi_j )  \big) \, .
\end{aligned}
$$
Let us   estimate~$     \Lambda^{-1}_{1,n}H_{n,N}^{(1)}$: we notice that~$\displaystyle  \sum_{j =1 }^J  e^{t \Delta}   \Lambda_{j,n}\phi_j =   e^{t \Delta} \sum_{j =1 }^J  \Lambda_{j,n}\phi_j $ so
 $$
\begin{aligned}
\big \|      \Lambda^{-1}_{1,n}H_{n,N}^{(1)}\big\|_{{\mathscr L}^{1:\infty}_{q}(\infty)} & =
\big \|H_{n,N}^{(1)}\big\|_{{\mathscr L}^{1:\infty}_{q}(\infty)}
\\
& \lesssim
 \sum_{\ell = 1}^{N-1} \big\|  e^{t \Delta}\sum_{j =1 }^J \Lambda_{j,n}\phi_j \big\|_{{\mathscr L}^{1:\infty}_{q}(\infty)}^\ell \\
\end{aligned}
$$
by Lemma~\ref{decompositionNS}, hence by classical bounds on the heat flow we get
$$
\big \|      \Lambda^{-1}_{1,n}H_{n,N}^{(1)}\big\|_{{\mathscr L}^{1:\infty}_{q}(\infty)}\lesssim \sum_{\ell = 1}^{N-1} \big \| \sum_{j =1 }^J   \Lambda_{j,n}\phi_j \big\|_{\dot B^{s_q}_{q,q}}^\ell = \sum_{\ell = 1}^{N-1} \big \| u_{0,n}-\psi_n^J \big\|_{\dot B^{s_q}_{q,q}}^\ell \, .
$$
by~(\ref{profilesa}) with $f_n=u_{0,n}$.  Hence by (\ref{orth2a}) and our assumption on $(u_{0,n})$
  we find
$$
\big \|      \Lambda^{-1}_{1,n}H_{n,N}^{(1)}\big\|_{{\mathscr L}^{1:\infty}_{q}(\infty)} \leq C(N) \, .
$$
Since the term~$    \Lambda^{-1}_{1,n}H_{n,N}^{(2)}$ goes to zero in~${\mathscr L}^{a:\infty}_{q}(\infty)$ as~$n$ goes to infinity for fixed~$J$ thanks to Lemma~\ref{decompositionNS},
where~$a$ is any real number such that~$1-3/q < 1/a<1 $,
we infer   that
$$
\lim_{J \to \infty} \limsup_{n \to \infty}  \Big \| \sum_{j =1 }^J       \Lambda^{-1}_{1,n}
  \Lambda_{j,n}
H_{N,j}
\Big \|_{ {\mathscr L}^{a:\infty}_q(\infty)} \leq C(N) \, .
$$
\noindent
Finally we are left with the study of~$\displaystyle  \sum_{j =1 }^J       \Lambda^{-1}_{1,n}
  \Lambda_{j,n}
Z_{N,j}$.
We recall that thanks to~(\ref{ortho33}),
there is~$J_0 \in \N$ such that for all~$j \geq J_0$, $\|\phi_j\|_{\dot B^{s_q}_{q,q}} \leq c_0$. Then thanks to Lemma~\ref{decompositionNS} (1), we have
\begin{equation}\label{estimatejgeqj0}
\forall j \geq J_0 \, , \quad  \|Z_{N,j}\|_{ {\mathscr L}^\frac qN_{\frac qN}(\infty)} \lesssim \|\phi_j\|_{\dot B^{s_q}_{q,q}} ^N \, .
 \end{equation}
Now let us write,  for each~$J \geq J_0$,
 $$
\begin{aligned}
 \big \|\sum_{j=1}^J    \Lambda_{1,n}^{-1}     \Lambda_{j,n} Z_{N,j}\big\|_{ {\mathscr L}^\frac qN_{\frac qN}(T_1)}
   \leq   \big \|\sum_{j=1}^{J_0-1}      \Lambda_{1,n}^{-1}     \Lambda_{j,n} Z_{N,j}\big\|_{ {\mathscr L}^\frac qN_{\frac qN}(T_1)}   +  \big \|\sum_{j=J_0}^J       \Lambda_{1,n}^{-1}    \Lambda_{j,n} Z_{N,j}\big\|_{ {\mathscr L}^\frac qN_{\frac qN}(\infty)} \, .\end{aligned}
 $$
To control both terms on the right-hand side, we invoke~\cite[Lemma 3.6]{GKP}, according to which for any~$1 \leq J' \leq J$,
\begin{equation}\label{invoke6}
\begin{aligned}
\big \|\sum_{J'}^J   \Lambda_{1,n}^{-1}     \Lambda_{j_1,n} Z_{N,j}\big\|_{ {\mathscr L}^\frac qN_{\frac qN}(T_1)} ^\frac qN
=  \sum_{J'}^J  \|   \Lambda_{1,n}^{-1}     \Lambda_{j_1,n}   Z_{N,j}\|_{ {\mathscr L}^\frac qN_{\frac qN}(T_1)} ^\frac qN
+ \e (J,n) \\
  \mbox{where}\quad \forall J \in \N \, , \quad  \e (J,n)\to 0
 \, , \quad n \to \infty\, .
 \end{aligned}
\end{equation}
This gives on the one hand
$$
\begin{aligned}
 \big \|\sum_{j=1}^{J_0-1}    \Lambda_{1,n}^{-1}    \Lambda_{j,n} Z_{N,j}\big\|_{ {\mathscr L}^\frac qN_{\frac qN}(T_1)}^\frac q N & \leq \sum_{j=1}^{J_0-1} \|  \Lambda_{1,n}^{-1}
     \Lambda_{j,n}
 Z_{N,j}\|^\frac qN_{ {\mathscr L}^\frac qN_{\frac qN}(T_1)} + \e (J_0,n) \\
 & \leq   \sum_{j=1}^{J_0-1} \|   Z_{N,j}\|^\frac qN_{ {\mathscr L}^\frac qN_{\frac qN}(\lambda_{0,n}^2 \lambda_{j,n}^{-2} T_1)}   + \e (J_0,n)\\
 & \leq  \sum_{j=1}^{J_0-1} \|   Z_{N,j}\|^\frac qN_{ {\mathscr L}^\frac qN_{\frac qN}( T_j)}  + \e (J_0,n) \,,
 \end{aligned}
$$
the last line being due to~(\ref{eq:increasTglobal}). This implies that
$$
 \limsup_{n \to \infty} \big \|\sum_{j=1}^{J_0-1}    \Lambda_{1,n}^{-1}    \Lambda_{j,n} Z_{N,j}\big\|_{ {\mathscr L}^\frac qN_{\frac qN}(T_1)}^\frac q N <\infty \, .
$$
On the other hand we have, still thanks to~(\ref{invoke6}),
$$
\big \|\sum_{j=J_0}^{J}    \Lambda_{1,n}^{-1}     \Lambda_{j,n} Z_{N,j}\big\|_{ {\mathscr L}^\frac qN_{\frac qN}(\infty)}^ \frac q N \leq \sum_{j=J_0}^{J}  \| Z_{N,j}\|_{ {\mathscr L}^\frac qN_{\frac qN}(\infty)} ^\frac q N + \e (J,n) \, ,
$$
so by~(\ref{estimatejgeqj0}) we infer that
$$
 \big \|\sum_{j=J_0}^{J}    \Lambda_{1,n}^{-1}     \Lambda_{j,n} Z_{N,j}\big\|_{ {\mathscr L}^\frac qN_{\frac qN}(\infty)}^ \frac q N  \lesssim \sum_{j=J_0}^J\|\phi_j\|_{\dot B^{s_q}_{q,q}} ^q + \e (J,n) \, .
$$
Using~(\ref{ortho33}) we conclude that
$$
\lim_{J \to \infty} \limsup_{n \to \infty} \big \|\sum_{j=J_0}^{J}    \Lambda_{1,n}^{-1}     \Lambda_{j,n} Z_{N,j}\big\|_{ {\mathscr L}^\frac qN_{\frac qN}(\infty)}  <\infty  \, ,
$$
and this ends the proof of Lemma~\ref{lemmadrift}. \hfill $\Box$
\\\\
We now turn to the source term and prove  Lemma~\ref{lemmasource}.
\\\\
{\bf Proof of Lemma \ref{lemmasource}:} \quad  The proof of this result is
an improvement (thanks to Lemma~\ref{lemmadrift}) of the proof of the corresponding result in~\cite{GKP}, namely the proof of~\cite[Lemma 3.7]{GKP}. We shall therefore only detail the new arguments.
\\\\
On the one hand it is proved in~\cite{GKP}, thanks to elementary product laws, that
\begin{equation}\label{limQWW} \lim_{J \rightarrow \infty} \limsup_{n
\rightarrow \infty} \big \|Q ( W_n^{J,1}, W_n^{J,1}) \big\|_{\tl^{q'}(\R^+;\dot B^{s_q -  \frac2 {q}}_{q,q})} = 0 \, .
\end{equation}
The terms~$  Q( U_n^{j,1}  , U_n^{j',1})$ when~$j \neq j'$    are also estimated exactly as in~\cite{GKP} (see also~(\ref{estimateB2}) in this paper, which provides a more general result for~$\ell (\geq 2)$ profiles).
\\\\
Now let us consider the last term entering in the definition of~$G_n^{J,1}$, namely the term~$Q ( F_n^{J,1}, W_n^{J,1})$. Writing~$fg=\mathcal{T}_f g+\mathcal{T}_g f+\mathcal{R}(f,g)$
the paraproduct decomposition of the
product~$fg$, we have by product estimates (\ref{prodest}) and H\"older's inequality in time followed by Bernstein's inequalities (\ref{bernst}) that
$$
\big \| \mathcal{T}_{F_n^{J,1}} W_n^{J,1}\big \| _{\tl^{q'}([0,T_1];\dot B^{s_q -1+ \frac2 {q'}}_{q,q})}
\lesssim \|F_n^{J,1}\|_{\tl^{q}([0,T_1];\dot B^{s_q + \frac2 {q}}_{q,q})}
\|W_n^{J,1}\|_{\tl^{r_1}([0,T_1];\dot B^{s_q + \frac2 {r_1}}_{q,q})}
$$
with~$1/r_1+ 1/q = 1/q'$.
This holds because~$s_\infty + 2/q <0$.
\smallskip
\noindent Similarly, since $2s_q + 2/r_1 + 2/q = 4/q>0$,  (\ref{bernst}) followed by (\ref{prodest}) give
$$
\big \|\mathcal{R}(F_n^{J,1},W_n^{J,1})\big\|_{\tl^{q'}([0,T_1];\dot B^{s_q - 1+ \frac2 {q'}}_{q,q})}
 \lesssim\|F_n^{J,1}\|_{\tl^{q}([0,T_1];\dot B^{s_q + \frac2 {q}}_{q,q})}
\|W_n^{J,1}\|_{\tl^{r_1}([0,T_1];\dot B^{s_q + \frac2 {r_1}}_{q,q})} \, .
$$
It follows that
\begin{equation}\label{limTWU}
\begin{aligned} \lim_{J \rightarrow \infty} \limsup_{n
\rightarrow \infty}
\big \|\mathcal{T}_{F_n^{J,1}} W_n^{J,1} + \mathcal{R}(F_n^{J,1},W_n^{J,1})\big\|_{\tl^{q'}([0,T_1];\dot B^{s_q -1+ \frac2 {q'}}_{q,q})} =0 \, .
\end{aligned}
\end{equation}
In order to improve on~\cite[Lemma 3.7]{GKP}, the only term to study is~$\mathcal{T}_{ W_n^{J,1} }  F_n^{J,1}$. Thanks to Lemma~\ref{lemmadrift}   we know that one can write
$$
F_n^{J,1} = F_n^{J,1,1}+F_n^{J,1,2} \quad \mbox{with} \quad
$$
$$
\lim_{J \rightarrow \infty} \limsup_{n
\rightarrow \infty} \big(
 \|F_n^{J,1,1}\|_{  {\mathscr L}^{\frac qN}_{\frac qN}(T_1) } + \|F_n^{J,1,2}\|_{ {\mathscr L}^{a:\infty}_{q}(\infty) }  \big)
< \infty\, .
$$
Let us study~$F_n^{J,1,1}$. We define~$r_2$ by~$1/r_2 = (1-\delta)/2$, so that by paraproduct estimates (\ref{prodest})
 (thanks to the fact that~$\delta>0$) followed by (\ref{bernst}) to embed~$\dot B^{s_q + \frac2 {r_2}}_{q,q}$ into~$\dot B^{s_\infty + \frac2 {r_2}}_{\infty,q}$
  we get immediately
$$
\big \| \mathcal{T}_{ W_n^{J,1}} F_n^{J,1,1} \big \| _{\tl^{r}([0,T_1];\dot B^{s_q -1+  \frac2 {r}}_{q,q})}
\lesssim  \|F_n^{J,1,1}\|_{  {\mathscr L}^{\frac qN}_{\frac qN}(T_1) }
\|W_n^{J,1}\|_{\tl^{r_2}([0,T_1];\dot B^{s_q + \frac2 {r_2}}_{q,q})} \, .
$$
Similarly given~$\e>0$, define~$r_3$ by~$2/r_3 =1-\frac 3q-\e $ (so that $s_q+2/r_3 <0$). Then if~$r_4$ satisfies~$1/r_4 + 1/r_3 = 1/r$ (notice that~$1/r_4  < 1-3/q< 1/a$ ), by (\ref{bernst}) followed by (\ref{prodest}) we can estimate
 $$
 \begin{aligned}
\big \|  \mathcal{T}_{ W_n^{J,1}} F_n^{J,1,2}\big \| _{\tl^{r}([0,T_1];\dot B^{s_q -1+  \frac2 {r}}_{q,q})}&\lesssim  \|F_n^{J,1,2}\|_{  {\mathscr L}^{r_4}_{  q}(T_1) }
\|W_n^{J,1}\|_{\tl^{r_3}([0,T_1];\dot B^{s_q + \frac2 {r_3}}_{q,q})} \\
&\lesssim  \|F_n^{J,1,2}\|_{  {\mathscr L}^{a:\infty}_{  q}(T_1) }
\|W_n^{J,1}\|_{\tl^{r_3}([0,T_1];\dot B^{s_q + \frac2 {r_3}}_{q,q})} \, .
\end{aligned}
$$
Lemma~\ref{lemmasource} is proved.  \hfill $\Box$

\subsection{An elementary decomposition via
  iteration: proof of Lemma~\ref{decompositionNS}}\label{elemdecompsec}  \quad
The argument leading to the result in Lemma~\ref{decompositionNS} can
be found  in~\cite{gip3} (in turn inspired by \cite{Planchon}); we detail it here for the convenience of the reader. The idea is to expand the solution in Duhamel form
\begin{equation}\label{expansion}
u = u_L+ B_2(u,u)
\end{equation}
where~$u_L:=e^{t\Delta}u_0$ and
\begin{equation}\label{defB2}
B_2(u,v)(t)
:=- \frac12 \int_0^t e^{(t-t') \Delta} {\mathbb P} \mbox{div} \, (u \otimes v +v\otimes u ) (t') \, dt' \, .
\end{equation}
This gives the desired expansion when~$N=2$:
$$
u = H_2 + Z_2
$$
with
$$
H_2 := u_L \quad \mbox{and} \quad Z_2
:=B_2(u,u) \, .
$$
In particular~${\mathscr{B}}_1 \equiv \mbox{Id}$.
Classical estimates on the heat flow imply that~$H_2$ belongs to~${\mathscr L}^{1:
\infty}_{p}(\infty)$.
 Moreover   product laws in Besov spaces along with
 the same heat flow estimates
 imply that $$
\| a\otimes b \|_{ \tl^1_t(\dot B^{2s_p+2}_{\frac p2,\frac p2})}
\lesssim \|a\|_{{\mathscr L}^{1:
\infty}_{p}(\infty)}\|b\|_{{\mathscr L}^{1:
\infty}_{p}(\infty)} \, ,
$$
so
\begin{equation}\label{estimateZ2}
\|Z_2  \|_{ {\mathscr L}^{1:
\infty}_{\frac p2,\frac p2}(T)}\lesssim  \|u\|_{{\mathscr L}^{1:
\infty}_{p}(T)}^2 \, .
\end{equation}
Note that the fact that the bilinear term allows to pass from an~$L^p$ to an~$L^\frac p2$ integrability
is a key feature in the whole of this paper, and will actually be used also in the next section extensively.
Next  we plug   the expansion~(\ref{expansion}) of~$u$ into the term~$B_2(u,u) $, to find
$$
\begin{aligned}
u &  = u_L + B_2 (u ,u   )\\
&  = u_L + B_2\big(u_L + B_2(u,u),u_L + B_2(u,u) \big) \\
&  = u_L +B_2(u_L,u_L)+2  B_2\big(u_L ,  B_2(u,u) \big)  + B_2 \big(B_2(u,u) ,  B_2(u,u) \big) \, .
\end{aligned}
$$
This gives the expansion for~$N=3$:
$$
u = H_3+Z_3
\quad \mbox{with} \quad
 H_3 := H_2 +B_2(u_L,u_L)
$$
so~${\mathscr{B}}_2\equiv \mbox{Id} + B_2$,
and
\begin{equation}\label{formulaH3}
\begin{aligned} Z_3
& :=  2  B_2 \big(u_L ,  B_2(u,u) \big) + B_2 \big(B_2(u,u) ,  B_2(u,u) \big)  \\
& =:    B_3 (u,u,u_L) + B_4 (u,u,u,u) \, ,
\end{aligned}
\end{equation}
and the expected bounds follow again from product laws as soon as~$p/2 > 3$.
Iterating further, the formulas immediately get very long and complicated so let us now argue by induction: assume that
$$
u = H_N+Z_N
$$
with~$H_N$  the sum of a finite number of multilinear operators of
order at most~$N-1$, acting on~$u_L $ only, and where~$Z_N$ has the following property:  we assume
there is an integer~$K_N \geq 0$, and for all~$0 \leq k \leq K_N$
some~$(N+k)$-linear operators~$B^M_{N+k} $ (the parameter~$M \in
\{0,\dots , N+k\}$ measures the number of entries in which $u$, rather
than $u_L$, appears), such that~$Z_N$ may be written in the form
\begin{equation}\label{writeZN}
Z_N =
 \sum_{M=1}^{N} B^M_{N} (u^{\otimes M} , u_L^{\otimes (N-M)})
+\sum_{k=1}^{K_N} \sum_{M=0}^{N+k} B^M_{N+k} (u^{\otimes M} , u_L^{\otimes (N+k-M)})
\, ,
\end{equation}
where we have used the following convention:~$ B_{N}^M  (u ^{\otimes M}, v^{\otimes (N-M)}
)$ denotes an ~$N$-linear operator~$ B_{N}^M  $   applied to~$M $ copies of a function~$u$ and~$(N-M)$ copies of a function~$v$:
$$
 B_{N}^M   (\underbrace {u,\dots,u}_{\mbox{\footnotesize{$M$ terms}}}
\, ,\underbrace {v,\dots,v}_{\mbox{\footnotesize{$N\!\!-\!\!M$ terms}}}) = B_{N}^M   (u ^{\otimes M}, v^{\otimes (N-M)}
)¬†\, .
$$
This notation is equivocal since the operator~$B_N$ need not be symmetric, but it will suffice for our purposes.
So  let us prove that for any~$M \geq 1$ and any~$N \in \N$, one can further decompose
\begin{equation}\label{enoughtoprove}
 B_{N}^M  (u ^{\otimes M}, u_L^{\otimes (N-M)}) = B^M_{N}(u_L^{\otimes N}) + Z_{N+1}
\end{equation}
where~$Z_{N+1}$ may be written in the following way, similarly to~(\ref{writeZN}): there is an integer~$K_{N+1}\geq 0$ and
 for all~$0 \leq k \leq K_{N+1}$ and~$0 \leq M \leq N+1+k$, some~$N+1+k$-linear operators~$\widetilde  B^M_{N+1+k} $ such that
$$
Z_{N+1} =  \sum_{M=1}^{N} \widetilde B^M_{N+1} (u^{\otimes M}
,u_L^{\otimes(N+1-M)})+ \sum_{k=0}^{K_{N+1}} \sum_{M=0}^{N+k} \widetilde B^M_{N+1+k} (u^{\otimes M}
,u_L^{\otimes(N+1+k-M)}) \, .
$$
This will imply that~$H_{N} \equiv H_{N-1} + B_N $, with~$B_N (u_L) =  B^M_{N}(u_L^{\otimes N}) $ an ~$N$-linear operator in~$u_L$.
In order to prove~(\ref{enoughtoprove})  we just need to use~(\ref{expansion}) again: replacing~$u$ by~$u_L + B_2(u,u)$ in the argument of~$B_N^M$ in~(\ref{writeZN}) gives
$$
\begin{aligned}
 B_N ^M (u ^{\otimes M}, u_L^{\otimes (N-M)})
& =B_N  ^M\Big (\big(u_L + B_2(u,u) \big) ^{\otimes M}, u_L^{\otimes (N-M)} \Big) \\
& =  B_N ^M(u_L^{\otimes N}) + \sum_{\ell = 1}^M \widetilde B^M_{N+\ell} (u^{\otimes 2\ell}, u^{\otimes (N-\ell)})
 \end{aligned}
$$
which proves~(\ref{enoughtoprove}).
To conclude  the proof of the first part of the lemma it remains to prove that~$H_N \in {\mathscr L}^{1:
\infty}_{p}(\infty)$ and~$Z_N \in \cl^{p/N}_{p/N}(T)$,
which again follows from product laws as long as $p>3(N-1)$.

\medskip
\noindent The proof of results~(1) to~(3) follows from the above construction as follows:

\medskip
\noindent The first result follows from~(\ref{estimateZ2}): if~$\|u_0\|_{\dot B^{s_p}_{p,p}} \leq c_0$, then by small data theory we have that~$T = \infty$ and
$$
 \|u\|_{{\mathscr L}^{1:
\infty}_{p}(\infty)} \leq 2  \|u_0\|_{\dot B^{s_p}_{p,p}}  \, .
$$
So~(\ref{estimateZ2}) becomes
$$
\|Z_2  \|_{ {\mathscr L}^{1:
\infty}_{\frac p2,\frac p2}(\infty)}¬†\lesssim  \|u\|_{{\mathscr L}^{1:
\infty}_{p}(\infty)}^2  \lesssim \|u_0\|_{\dot B^{s_p}_{p,p}} ^2 \, .
$$
The argument is the same at each step of the construction of~$Z_N$, since
$$
\|u_L\|_{{\mathscr L}^{1:
\infty}_{p}(\infty)} \lesssim \|u_0\|_{\dot B^{s_p}_{p,p}} \, .
$$

\smallskip
\noindent
The second result follows from scale invariance of the Navier-Stokes equations, hence of~$B_2$ defined in~(\ref{defB2}), and the iterative construction of~${\mathscr{B}}_N$.

\medskip
\noindent
Let us prove the last result. We shall only prove the more difficult result~(\ref{estimateB2}), as~(\ref{estimateB1}) follows from the same estimates.
 We shall detail the argument for~${\mathscr{B}}_2$ and~${\mathscr{B}}_3$,  and then show how to
pursue the computation for higher orders.

\medskip
\noindent $ \bullet $ $ $   We recall that~${\mathscr{B}}_2 \equiv \mbox{Id} + B_2$ with~$B_2$ defined in~(\ref{defB2}), and heat flow estimates imply that
\begin{equation}\label{estimateHn21}
\| {  B}_{2}\big(   \Lambda_{j_1,n} U_1,  \Lambda_{j_2,n} U_2 \big) \| _{{\mathscr L}^{a:\infty}_{p}(\infty)}  \lesssim \|    \Lambda_{j_1,n} U_1 \otimes   \Lambda_{j_2,n} U_2  \|_{\tl^a(\R^+;\dot B^{-2+\frac 3p+\frac2a}_{p,p})}  \, .
\end{equation}
Notice that by density in~${\mathscr L}^{1:\infty^-}_{p}(\infty)$, where $\infty^-$ indicates any arbitrarily large but finite number, we can assume that~$U_1 $ and~$U_2$ are smooth and compactly supported in space-time.
More precisely: given~$\e>0$ one can find two compactly supported (in space and time) functions~$U_1 ^\e$ and~$U_2^\e$ such that
\begin{equation}\label{approxdensity}
\| U_1 ^\e - U_1 \|_{{\mathscr L}^{1:\infty^-}_{p}(\infty)}   + \| U_2 ^\e - U_2 \|_{{\mathscr L}^{1:\infty^-}_{p}(\infty)} \leq \e \, .
\end{equation}
Product rules (along with the scale invariance of the scaling operators) give  for integers~$j , j'\in \{1,2\}$
$$
 \|    \Lambda_{j_1,n} (U_j ^\e - U_j) \otimes   \Lambda_{j_2,n} U_{j'}  \|_{\tl^a(\R^+;\dot B^{-2+\frac 3p+\frac2a}_{p,p})}+ \|    \Lambda_{j_1,n} (U_j ^\e - U_j) \otimes   \Lambda_{j_2,n} (U_{j'}^\e - U_{j'})  \|_{\tl^a(\R^+;\dot B^{-2+\frac 3p+\frac2a}_{p,p})}  \lesssim \e  \,  ,
$$
so let us now concentrate on the study of
$$
U_{n}^\e :=     \Lambda_{j_1,n}  U_1^\e \otimes     \Lambda_{j_2,n}  U_2^\e \, .
$$
We shall start by proving that as~$n$ goes to infinity,
\begin{equation}\label{resultdifferentscales}
\lambda_{j_1,n} / \lambda_{j_2,n}  + \lambda_{j_2,n} / \lambda_{j_1,n} \to \infty \quad \Longrightarrow
 \| U_{n}^\e\|_{\tl^a(\R^+;\dot B^{-2+\frac 3p+\frac2a}_{p,p})} \to 0  \, .
\end{equation}
Product laws and embeddings
give
\begin{equation}\label{estimateHn22}
\begin{aligned}
  \| U_{n}^\e  \|_{\tl^a(\R^+;\dot B^{-2+\frac 3p+\frac2a}_{p,p})}  \lesssim \|     \Lambda_{j_1,n}   U_1^\e \|_{\tl^{\tilde a}(\R^+;\dot B^{s}_{p,p})}
 \|      \Lambda_{j_2,n}   U_2^\e \|_{\tl^{\tilde a'}(\R^+;\dot B^{s }_{p,p})}
 \end{aligned}
\end{equation}
for any $\tilde a \geq a$ with $\frac 1{\tilde a} + \frac 1{\tilde a'} =1$ and $s=-1+3/p+1/a$. Notice that the product law is allowed thanks to condition~$1>1/a> 1- 3/p$, which implies~$2s>0$ and~$s-3/p<0$.
But an easy computation shows that (for each~$\e$)
$$
 \|   \Lambda_{j_1,n}   U_1^\e\|_{\tl^{\tilde a}(\R^+;\dot B^{s}_{p,p})}  \lesssim \lambda_{j_1,n}^{\frac 3p+\frac2{\tilde a} - s-1} =  \lambda_{j_1,n}^{\frac 2{\tilde a}-\frac 1 a}
$$
and
$$
 \|  \Lambda_{j_2,n}   U_2^\e \|_{\tl^{\tilde a'}(\R^+;\dot B^{s}_{p,p})}
 \lesssim \lambda_{j_2,n}^{\frac 3p+\frac2{\tilde a'} - s-1} =  \lambda_{j_2,n}^{\frac 2{\tilde a'}-\frac 1 a}
 =  \lambda_{j_2,n}^{\frac 1 a-\frac 2{\tilde a}}
$$
so going back to~(\ref{estimateHn22}) we find that
$$
 \|U_{n}^\e\|_{\tl^a(\R^+;\dot B^{-2+\frac  3 p+\frac2a}_{p,p})} \leq \left(\frac{ \lambda_{j_1,n}}{ \lambda_{j_2,n}}\right)^{\frac 2{\tilde a}-\frac 1a} \to 0 \, ,  \quad n \to \infty\, ,
$$
if~$\lambda_{j_1,n} / \lambda_{j_2,n} \to 0$ as long as $a\leq \tilde a < 2a$.
Exchanging~$j_1$ and~$j_2$ in the computation if~$\lambda_{j_1,n} / \lambda_{j_2,n} \to \infty$, we conclude that~(\ref{resultdifferentscales}) holds.

\medskip
\noindent Now let us assume that~$ \lambda_{j_1,n} \equiv \lambda_{j_2,n}  $. Then by orthogonality of the cores and scales, we know that
$$
|x_{j_1,n}-x_{j_2,n}|/ \lambda_{j_1,n} \to \infty  \, , \quad n \to \infty\, .
$$
Then by scale and translation invariance we have
$$
 \| U_{n}^\e\|_{\tl^a(\R^+;\dot B^{-2+\frac 3 p+\frac2a}_{p,p})} = \Big \|
U_1^\e (x,t)\otimes U_2^\e\big (x  + \frac{x_{j_1,n}-x_{j_2,n}} {\lambda_{j_1,n} },t\big)
 \Big\|_{\tl^a(\R^+;\dot B^{-2+\frac 3 p+\frac2a}_{p,p})} \, .
 $$
Define
$$
\widetilde U_{n}^\e(t,x):=  U_1^\e (t,x)\otimes U_2^\e\big (x  + \frac{x_{j_1,n}-x_{j_2,n}} {\lambda_{j_1,n} },t\big) \, .
$$
Then due to the assumption on the supports of~$U_1^\e$ and~$U_2^\e$ we   find that
$$
|x_{j_1,n}-x_{j_2,n}|/ \lambda_{j_1,n} \to \infty  \quad  \Longrightarrow \quad \widetilde U_{n}^\e\equiv  0
$$
for~$n$ large enough, uniformly in~$x$ and~$t$.
With~(\ref{resultdifferentscales}) we therefore infer that  as soon as~$j_1 \neq j_2$ then
$$
 \|      \Lambda_{j_1,n} U_1^\e\otimes     \Lambda_{j_2,n} U_2^\e \|_{\tl^a(\R^+;\dot B^{-2+\frac 3 p+\frac2a}_{p,p})} \to 0 \, ,\quad  n \to \infty\, .
$$
Plugging that result into~(\ref{estimateHn21})  and recalling~(\ref{approxdensity}), we find that
\begin{equation}\label{thefinalresultB2}
\begin{aligned}
& (\lambda_{j_1,n} , x_{j_1,n}) \perp  (\lambda_{j_2,n} , x_{j_2,n})  \quad \mbox{and} \quad  U_1, U_2 \in {\mathscr L}^{1:\infty}_{p}(\infty)\\
& \quad   \Longrightarrow\quad \| {  B}_{2}\big(   \Lambda_{j_1,n} U_1,  \Lambda_{j_2,n} U_2 \big) \| _{{\mathscr L}^{a:\infty}_{p}(\infty)}   \to 0 \, , \quad n \to \infty \, .
\end{aligned}
\end{equation}

\medskip
\noindent $ \bullet $ $ $ Next let us consider~${\mathscr{B}}_3 $.
We recall that~${\mathscr{B}}_3 = \mbox{Id} + B_2 + B_3$ where from~(\ref{formulaH3}) we can recover
$$
B_3 (U_1,U_2,U_3) = 2B_2 \big(U_1, B_2(U_2,U_3) \big)  \, .
$$
Now let~$ (\lambda_{j_1,n} , x_{j_1,n}) ,  (\lambda_{j_2,n} , x_{j_2,n}),  (\lambda_{j_3,n} , x_{j_3,n})  $ be a set of scales and cores, such that at least two are orthogonal. We write
$$
 {  B}_{3}\big(  \Lambda_{j_1,n} U_1,   \Lambda_{j_2,n} U_2,  \Lambda_{j_{3},n} U_{3} \big)
= 2B_2 \big(  \Lambda_{j_1,n}U_1, B_2(  \Lambda_{j_2,n} U_2,  \Lambda_{j_{3},n} U_{3}) \big)
$$
and let us start by assuming that~$ (\lambda_{j_2,n} , x_{j_2,n})$ is orthogonal to~$ (\lambda_{j_3,n} , x_{j_3,n}) $. Then
  we simply write   by product laws again,
$$\begin{aligned}
& \big\| B_2 (  \Lambda_{j_1,n} U_1, B_2(  \Lambda_{j_2,n} U_2,  \Lambda_{j_{3},n} U_{3}))
\big \| _{{\mathscr L}^{a:\infty}_{p}(\infty)}\\
& \quad \lesssim \| U_1 \| _{{\mathscr L}^{1:\infty}_{p}(\infty)}
\|  B_2(  \Lambda_{j_2,n} U_2,  \Lambda_{j_{3},n} U_{3})\|_{{\mathscr L}^{a:\infty}_{p}(\infty)}
\end{aligned}
$$
and we conclude as above thanks to~(\ref{thefinalresultB2}).

\smallskip
\noindent Conversely if~$ (\lambda_{j_2,n} , x_{j_2,n})$ is not orthogonal to~$ (\lambda_{j_3,n} , x_{j_3,n}) $ then without loss of generality we may assume that~$  \Lambda_{j_2,n}  \equiv   \Lambda_{j_{3},n} $, and~$(\lambda_{j_1,n} , x_{j_1,n})$ must be orthogonal to~$(\lambda_{j_2,n} , x_{j_2,n})$.
We therefore have
$$
B_2 \big(  \Lambda_{j_1,n} U_1, B_2(  \Lambda_{j_2,n} U_2,  \Lambda_{j_{3},n} U_{3}) \big) =
B_2 \big (  \Lambda_{j_1,n} U_1,   \Lambda_{j_{2},n}  B_2(   U_2,U_{3}) \big)  \, .
$$
Defining~$  U_2:= B_2(  U_2,U_{3})$ we know that
$$
\|   U_{2}\|_{{\mathscr L}^{1:\infty}_{p}(\infty)}
\lesssim \|  U_2\|_{{\mathscr L}^{1:\infty}_{p}(\infty)}  \|  U_{3}\|_{{\mathscr L}^{1:\infty}_{p}(\infty)}
$$
so we can conclude  again using~(\ref{thefinalresultB2}).

 \medskip
\noindent $ \bullet $ $ $ In the case of higher order operators~$B_\ell$, with~$\ell \geq 4$, we apply exactly the same strategy as above: by construction,~$B_\ell$ writes as a bilinear operator~$B_2$ whose arguments are either~$u_L,$ ~$B_2(u_L,u_L)$, or iterates of those bilinear operators like~$B_2(u_L,B_2(u_L,u_L))$ and so forth.

\smallskip
\noindent
If in  the formula defining~$B_\ell$, two orthogonal vector fields~$  \Lambda_{j_k,n} U_{k}$ and~$  \Lambda_{j_{k'},n} U_{k'}$ appear as  the two arguments of an operator~$B_2$, as in~$B_2(  \Lambda_{j_k,n} U_{k},  \Lambda_{j_{k'},n} U_{k'})$, then we use product laws to find
$$
\big \|  {  B}_{\ell}\big(   \Lambda_{j_1,n} U_1, \dots,     \Lambda_{j_\ell,n} U_\ell \big)\big \| _{{\mathscr L}^{a:\infty}_{p}(\infty)}  \lesssim \|B_2\big (   \Lambda_{j_k,n} U_{k} ,   \Lambda_{j_{k'},n} U_{k'}\big)\|_{{\mathscr L}^{a:\infty}_{p}(\infty)}
$$
and we conclude with~(\ref{thefinalresultB2}) again.

\smallskip
\noindent If that is not the case, that means that
each time an operator~$B_2 (  \Lambda_{j_i,n}    U_i ,  \Lambda_{j_{i'},n}    U_{i'} )$ appears in~$B_\ell$, then again without loss of generality we may assume~$  \Lambda_{j_i,n}  \equiv   \Lambda_{j_{i'},n}$ so we can unscale   that~$B_2$ operator using
$$
B_2 (  \Lambda_{j_{i},n}    U_i ,  \Lambda_{j_{i},n}    U_{i'} ) =   \Lambda_{j_{i},n}       U_{i,i'} :=    \Lambda_{j_i,n}  B_2 (   U_i ,     U_{i'} ) \, .
$$
Then we iterate this procedure, noticing that~$  U_{i,i'} $ is independent of~$n$ and belongs to~${\mathscr L}^{1:\infty}_{p}$ by product laws.
At some stage of the procedure,
since by assumption some scales are orthogonal, one ends up in  a situation where in  the formula defining~$B_\ell$, there appears a term of the form~$B_2(   \Lambda_{j_k,n}    U _k ,  \Lambda_{j_{k'},n}   U_{k'}  )$ with~$  \Lambda_{j_k,n} $ and~$  \Lambda_{j_{k'},n}$ orthogonal and where~$   U_k  $ and~$  U_{k'} $ depend on other functions~$U_j$ via a possibly large number of iterations of operators~$B_2$, but are independent of~$n$ and belong to~${\mathscr L}^{1:\infty}_{p}(\infty)$ as in the previous case. So again we can use~(\ref{thefinalresultB2}) and the result follows.

 \medskip
\noindent The lemma is proved.
\hfill $\Box$

\subsection{An orthogonality result: proof of Proposition \ref{claim33}}\label{proofclaim33}
\quad
Let us define
$$
v_n:=
u_n (t_n)-\big(  \Lambda_{1,n}U_1 \big) (t_n)  \, .
$$
Then
$$
\begin{aligned}
\|u_n(t_n)\|_{\bpp}^p & = \sum_{j \in \Z} 2^{jp s_{p}} \|\Delta_j u_n(t_n)\|_{L^p}^p\\
& = \sum_{j \in \Z} 2^{jp s_{p}}  \Big\|\Delta_j  \Big(\big(  \Lambda_{1,n}U_1 \big) (t_n) +v_n \Big) \Big \|_{L^p}^p \, ,
\end{aligned}
$$
and we now decompose
\begin{equation}
\label{decomposition1}\begin{aligned}
&  \sum_{j \in \Z} 2^{jp s_{p}} \Big\|\Delta_j  \Big(\big(  \Lambda_{1,n}U_1 \big) (t_n) +v_n \Big) \Big \|_{L^p}^p - \big\|\big(  \Lambda_{1,n}U_1 \big) (t_n) \big\|_{\bpp}^p -
\|v_n(t_n)\|_{\bpp}^p \\
& \quad =
\sum_{r=1}^{p-1}C_p^r \int_{\R^d}
\sum_{j \in \Z} 2^{jp s_{p}}    \big | \Delta_j \big(  \Lambda_{1,n}U_1 \big) (t_n) \big |^r  \big | \Delta_j v_n\big |^{p-r} (x)\,  dx \, .
\end{aligned}
\end{equation}
To prove the lemma, it therefore suffices to prove that for all~$1 \leq r \leq p-1$,
\begin{equation}
\label{decomposition}
\lim_{J \to\infty} \limsup_{n \to \infty} \int_{\R^d}
\sum_{j \in \Z} 2^{jp s_{p}}    \big | \Delta_j \big(  \Lambda_{1,n}U_1 \big) (t_n) \big |^r  \big | \Delta_j v_n\big |^{p-r} (x)\,  dx = 0 \, .
\end{equation}
Next let us write, for a given~$J$ (large)
$$
v_n = v_n^{(1,J)} + v_n^{(2,J)}  \, ,
$$
where we have defined
\begin{equation}\label{defvn1vn2}
\begin{aligned}
 v_n^{(1,J)}   := \sum_{k=2}^{J} \big(  \Lambda_{k,n}U_k\big)(t_n )  \quad \mbox{and}¬†\quad
  v_n^{(2,J)}  := w_n^J(t_n)+  r_n^J(t_n)    \, .
 \end{aligned}
\end{equation}
 First let us   deal with the contribution of~$v_n^{(2,J)}$, which is the easiest: recalling that~$w_n^J(t)= e^{t\Delta} \psi_n^J$,
we   start by noticing that
\begin{equation}
\label{limitzero}
\begin{aligned}
\|v_n^{(2,J)}\|_{\dot B^{-1}_{\infty,\infty}} & \leq \| e^{t_n\Delta}   \psi_n^J\|_{\dot B^{-1}_{\infty,\infty}} +
\|r_n^J(t_n)\|_{\dot B^{-1}_{\infty,\infty}} \\
&\lesssim  \|    \psi_n^J\|_{\dot B^{s_q}_{q,q}} +
\|r_n^J(t_n)\|_{\dot B^{s_q}_{q,q}} \to 0 \, , \quad J \to \infty \,,
\end{aligned}
\end{equation}
uniformly in~$n$. But by H\"older's inequality in the~$x$ variable,
$$
\begin{aligned}
\int_{\R^d}
\sum_{j \in \Z} 2^{jp s_{p}}   &  \big | \Delta_j \big(   \Lambda_{1,n}U_1 \big) (t_n) \big |^r  \big | \Delta_j v_n^{(2,J)}\big |^{p-r} (x)\,  dx \\
&  \leq \sum_{j \in \Z} 2^{jr s_r}    \big \| \Delta_j \big(  \Lambda_{1,n}U_1 \big) (t_n) \big \|_{L^r}^r \| \Delta_j v_n^{(2,J)}\|_{L^\infty} ^{p-r}
\end{aligned}
$$
so
$$
\begin{aligned}
\int_{\R^d}
\sum_{j \in \Z} 2^{jp s_{p}}   &  \big | \Delta_j \big(   \Lambda_{1,n}U_1 \big) (t_n) \big |^r  \big | \Delta_j v_n^{(2,J)}\big |^{p-r} (x)\,  dx \\
&  \leq \sum_{j \in \Z} 2^{jr s_r}    \big \| \Delta_j \big(  \Lambda_{1,n}U_1 \big) (t_n) \big \|_{L^r}^r  \|v_n^{(2,J)}\|_{\dot B^{-1}_{\infty,\infty}} ^{p-r} \\
&  \leq  \big\|  \big(  \Lambda_{1,n}U_1 \big) (t_n) \big\|_{\dot B^{s_r}_{r,r}} ^r \|v_n^{(2,J)}\|_{\dot B^{-1}_{\infty,\infty}} ^{p-r} \, .
\end{aligned}
$$
By scale invariance we find
$$
\begin{aligned}
\int_{\R^d}
\sum_{j \in \Z} 2^{jp s_{p}}      \big | \Delta_j \big(  \Lambda_{1,n}U_1 \big) (t_n) \big |^r  \big | \Delta_j v_n^{(2,J)}\big |^{p-r} (x)\,  dx  \lesssim\big\| U_1 (s)\big\|_{\dot B^{s_r}_{r,r}} ^r \|v_n^{(2,J)}\|_{\dot B^{-1}_{\infty,\infty}} ^{p-r} \, .
\end{aligned}
$$
It follows from~(\ref{limitzero}) that
\begin{equation}\label{resultvn2j}
\lim_{J \to \infty} \limsup_{n \to \infty } \int_{\R^d}
\sum_{j \in \Z} 2^{jp s_{p}}      \big | \Delta_j \big(  \Lambda_{1,n}U_1 \big) (t_n) \big |^r  \big | \Delta_j v_n^{(2,J)}\big |^{p-r} (x)\,  dx  = 0 \, .
\end{equation}

\medskip
\noindent
Let us now analyze the contribution of   the term~$v_n^{(1,J)}$, for~$J $ fixed. By orthogonality, as in~(\ref{orthogonalitybesovprofiles}), we know that that for any~$J' \leq J$,
\begin{equation}\label{orthoJJ'}
\big \| \sum_{k=J'}^{J} \big(  \Lambda_{k,n}U_k\big)(t_n )\big\|_\bpp^p = \sum_{k=J'}^{J} \| \big(  \Lambda_{k,n}U_k\big)(t_n )\|_\bpp^p + \varepsilon (J,n)
\end{equation}
where for each given~$J$, $\varepsilon (J,n)
\to 0$ as~$n \to \infty$.

\medskip
\noindent We notice that  each profile~$U_k$ may be chosen as smooth as
necessary in~$t$ and~$x$ (see e.g. \cite{Gallagher} for a similar procedure). Since by definition of~$t_n$
$$
 \big(  \Lambda_{k,n}U_k\big)(t_n ) =\frac 1 {\lambda_{k,n}} U_k \big( \frac{x-x_{k,n}}{\lambda_{k,n}} , \frac{\lambda_{1,n}^2s}{\lambda_{k,n}^2}\big)
$$
we get
$$
 \big\|
 \big(  \Lambda_{k,n}U_k\big)(t_n )
\big \|_{\bpp}  =  \big\|
 U_k\big(\cdot,\frac{\lambda_{1,n}^2s}{\lambda_{k,n}^2} \big) \big \|_{\bpp}
$$
hence in particular, by~(\ref{limTinfty}),
\begin{equation}\label{lambda0kinfty}
\lambda_{1,n}/\lambda_{k,n}\to \infty \quad \Longrightarrow \quad  \big\|
 \big(  \Lambda_{k,n}U_k\big)(t_n )
\big \|_{\bpp} \to 0 \, , \quad n \to \infty \, .
\end{equation}
Notice that~$\lambda_{1,n}/\lambda_{k,n}\to \infty $ is only possible if~$T_k^* = \infty $, by~(\ref{reorder}). From~(\ref{lambda0kinfty}) we get that for each~$J$,
$$
  \limsup_{n \to \infty } \Big(  \big \|  \sum_{k=1}^{J } \big(  \Lambda_{k,n}U_k\big)(t_n )  \big\|_\bpp^p - \Big \|  \sumetage{k=1}{\lambda_{1,n}/\lambda_{k,n}\to 0}^{J }   \Lambda_{k,n}\phi_k    \Big\|_\bpp^p - \Big \|  \sumetage{k=1}{\lambda_{1,n}\equiv\lambda_{k,n} }^{J}   \Lambda_{k,n}\big( U_k   (s) \big) \Big\|_\bpp^p  \Big) = 0 \, .
$$
It follows that to end the study of the contribution of~$ v_n^{(1,J)}$ we just need to prove the two following properties:  for each~$
1 \leq k \leq J$, if~$\lambda_{1,n}/\lambda_{k,n}\to 0$ then
\begin{equation}\label{limitphik}
\int_{\R^d}
\sum_{j \in \Z} 2^{jp s_{p}}    \big | \Delta_j \big(  \Lambda_{1,n}U_1 \big) (t_n) \big |^r  \big | \Delta_j   \Lambda_{k,n}\phi_k \big |^{p-r} (x)\,  dx \to 0 \, , \quad n \to \infty\, ,
\end{equation}
while if~$\lambda_{1,n}\equiv\lambda_{k,n}$ then
\begin{equation}\label{limitUk}
\int_{\R^d}
\sum_{j \in \Z} 2^{jp s_{p}}    \big | \Delta_j \big(  \Lambda_{1,n}U_1 \big) (t_n) \big |^r  \big | \Delta_j     \Lambda_{k,n}U_k   (s) \big |^{p-r} (x)\,  dx \to 0 \, , \quad n \to \infty \, .
\end{equation}
Let us start by proving~(\ref{limitphik}).
By density  we  assume that for each~$1 \leq k \leq J$,~$\phi_k$ has a spectrum restricted to a given ring of~$\R^d$, of small radius~$r_k$ and large radius~$R_k$.
It is plain to see that for any function~$f$,
$$
\Delta_j   \Lambda_{k,n}   f=    \Lambda_{k,n}  \Delta_{j+ \log_2\lambda_{k,n}}f
$$
so there are universal constants~$c$ and~$C$ such that
$$
\Delta_j   \Lambda_{k,n}  \phi _k  \neq 0 \Longrightarrow c r_k \leq 2^j \lambda_{k,n}
\leq CR_k  \, .
$$
Assuming similarly that~$ U_1 (s)  $ has  a spectrum restricted to a given ring of~$\R^d$, of small radius~$r_0$ and large radius~$R_0$, we get
$$
\Delta_j  \Lambda_{1,n}  U_1 (s)   \neq 0 \Longrightarrow c r_0 \leq 2^j \lambda_{1,n}
\leq CR_0\, .
$$
If~$\lambda_{1,n}/\lambda_{k,n}\to 0$ then those two conditions are asymptotically incompatible, hence
 $$
\begin{aligned}
&\sum_{j \in \Z}  \int_{\R^d}
  2^{jp s_{p}}    \big |\Delta_j \big(  \Lambda_{1,n}  U_1 \big) (t_n) \big |^r  \big | \Delta_j (    \Lambda_{k,n}  \phi_k ) \big |^{p-r} (x)\,  dx \to 0 \, , \quad n \to \infty \, ,
  \end{aligned}
 $$
 which proves~(\ref{limitphik}). Now let us prove~(\ref{limitUk}). If~$ \lambda_{1,n} \equiv  \lambda_{k,n}$
 then
 $$
\begin{aligned}
&\sum_{j \in \Z}  \int_{\R^d}
  2^{jp s_{p}}    \big |\Delta_j   \Lambda_{1,n}  U_1 (s,x) \big |^r  \big | \Delta_j \Lambda_{k,n}  U_k(s,x)\big |^{p-r}  \,  dx \\
 & = \sum_{j \in \Z}  \int_{\R^d}
  2^{jp s_{p}}    \big |\Delta_j      U_1  (s,y) \big |^r  \big | \Delta_j      U_k  (s, y + \frac{x_{1,n}-x_{k,n} }{ \lambda_{k,n}} )\big |^{p-r} (x)\,  dx
  \end{aligned}
 $$
 which goes to zero by Lebesgue's theorem, due to the orthogonality of the cores of concentration. So we have proved that
 \begin{equation}
\label{limitzerovn1}
 \int_{\R^d}
 \sum_{j \in \Z} 2^{jp s_{p}}    \big | \Delta_j \big(  \Lambda_{1,n}U_1 \big) (t_n) \big |^r  \big | \Delta_j v_n^{(1,J)}\big |^{p-r} (x)\,  dx  \to 0 \, ,\quad  n \to \infty \, .
\end{equation}

\medskip
\noindent
With~(\ref{resultvn2j}) this proves~(\ref{decomposition}) hence thanks to~(\ref{decomposition1}),
$$
\begin{aligned}
 \sum_{j \in \Z} 2^{jp s_{p}} \Big\|\Delta_j  \Big(\big( \Lambda_{1,n}U_1 \big) (t_n) +v_n \Big) \Big \|_{L^p}^p - \big\|\big( \Lambda_{1,n}U_1 \big) (t_n) \big\|_{\bpp}^p\\
  -
\|v_n(t_n)\|_{\bpp}^p \to 0 \, , \,  n \to \infty
\end{aligned}
$$
whence the result. \hfill $\Box$

\section{Improving bounds for solutions to Navier-Stokes via iteration}\label{iterationsection}
\noindent
The goal of this section is to prove Proposition  \ref{lemmaimprovedb}. This will follow, in Section~\ref{prooflemmaimproved}, from the following statement whose proof is postponed to  Section \ref{iterationprocedure}.
We define Kato spaces on a time interval $(0,T) $ for~$q \in (3,\infty]$ by
\begin{equation}\label{katospacedef}
\mathscr{K}_q(T):= \{u\in \mathcal{S}'(\rt \times \R^+) \, | \,
\|u\|_{\mathscr{K}_q(T)}:=\sup_{0<t \leq T} t^{-{s_q}/2}\|u(t)\|_{L^q} < \infty \}
\end{equation}
 as well as
 \begin{equation}\label{katospacedef1}
 \mathscr{K}^{1}_q(T):= \{u\in \mathcal{S}'(\rt \times \R^+) \, | \,
 \|u\|_{\mathscr{K}^{1}_q(T)}:=\sup_{0<t \leq T} t^{1/2-s_{q}/2}\|u(t)\|_{{\dot B}^{1}_{q,\infty}} < \infty \}
 \end{equation}
where we recall that $-s_q:=1-\frac 3 q>0$, and for  $3 < q_1 \leq q_2 \leq \infty$, we set
\begin{equation}\label{katospacedefrange}
\mathscr{K}_{q_1:q_2}(T):= \bigcap_{q_1 \leq q < q_2}  \mathscr{K}_q(T) \, .
\end{equation}
 \begin{remark}\label{usefulembedding}
Notice  that for $p>3$
\begin{equation}
  \label{eq:2}
  \|f\|^{}_{\mathscr{K}_p(T)} \lesssim \| f\|_{L^{\infty}(0,T;\dot
    B^{s_p}_{p,\infty})}^{\frac{p}{2p-3}}\|f\|_{\mathscr{K}^{1}_p(T)}^{\frac{p-3}{2p-3}}\,\,\text{
    and } \,\,     \|f\|^{}_{\mathscr{K}_\infty(T)} \lesssim \|f\|^{1-\frac3p}_{\mathscr{K}_p(T)} \|f\|^{\frac3p}_{\mathscr{K}^{1}_p(T)}
\end{equation}
which follow directly from the embeddings $\dot B^0_{q,1} \subset
L^q$ with $q=p,+\infty$ and $ \ell^{\infty}_{s_p}\cap \ell^{\infty}_{1}\subset \ell^{1}$, with $\|(\gamma_{j})_{j}\|_{ \ell^{\infty}_{s}} = \sup_{j} 2^{sj} |\gamma_{j}|$. Note moreover that
$\|e^{t\Delta}u_0\|_{\ck_p(\infty)}
+\|e^{t\Delta}u_0\|_{\ck^{1}_{p}(\infty)} \lesssim \|u_0\|_{\dot
  B^{s_p}_{p,\infty}}$ (in fact, both quantities on the left are
equivalent to the norm on the right).
\end{remark}
\pagebreak
\begin{thm}[Iteration and regularity of iterates]\label{expandprop}
Fix $p=3\cdot 2^{k}-2$ for some integer~$k$ such that~$k\geq 2$. Suppose ~$u_0 \in \dot B^{s_p}_{p,p}$, set $T^*:= T^*(u_0)$ and define the associated solution~$u:=NS(u_0) $ belonging to~$ \mathscr{L}^{1:\infty}_{p}[T<T^*]$.  Then there exists a family~$(u_{L,n})_{n\in [0,k]}$ (with~$u_{L,0}=\etl u_0$) and $w_{k}$ with the following three properties:
\begin{itemize}
\item[(I)]
$u_{L,n} \in \ln{1:\infty}_{\frac{p}{2^n},\frac{p}{2^n}}(T^*)$ for all $n\in \{0,\dots , k\}$, with
\begin{equation}
  \label{eq:6}
  \|u_{L,n}\|_{ \ln{1:\infty}_{\frac{p}{2^n},\frac{p}{2^n}}(T^*)}  \lesssim C(\|u_0\|_{ \dot B^{s_p}_{p,p}}) \,
\end{equation}
where $C$ is an explicit smooth function, and
\begin{equation}
  \label{eq:9}
u = \sum_{n=0}^{k} u_{L,n} + w_{k} \quad \textrm{in} \,\, \,\,
\mathscr{L}^{1:\infty}_{p}[T<T^*]\, .
\end{equation}
\item[(II)]
We have
$u_{L,n} \in  \mathscr{K}_{p:\infty}(T^*)$ for all $n\in \{0,\dots , k\}$.
\item[(III)]
If $u\in L^\infty (0,T^*;\dot B^{s_p}_{p,p})$, then
  $w_{k}$ has positive regularity up to time $T^*$, e.g. $w_{k} $ belongs to~$ \ln{\infty}_{\frac {6p}{2p+1},\infty}(T^*)$, and
\begin{equation}
  \label{eq:1}
  \|w_{k}\|_{ \ln{\infty}_{\frac
      {6p}{2p+1},\infty}(T^*)}\lesssim F(\|u_0\|_{\dot B^{s_p}_{p,p}}, \| u\|_{L^\infty (0,T^*;\dot B^{s_p}_{p,p})})\,,
\end{equation}
where $F$ is a  smooth function in two variables which may be computed explicitly.
    \end{itemize}
\end{thm}
\noindent
As in~(\ref{estimateZ2}) in the previous section,
the key argument to the proof of Theorem~\ref{expandprop}
is that the bilinear form allows improvement from~$L^\frac p{2^n}$ to~$L^\frac p{2^{n+1}}$ integrability: this allows one to show that each term in the expansion is smoother than the previous one, and to recover in a finite number of steps a positive regularity bound: this will be done in the next paragraph, while the proof of Theorem~\ref{expandprop} can be found in Section~\ref{iterationprocedure}.

\subsection{Negative regularity bounds to $L^p$ bounds: proof of Proposition \ref{lemmaimprovedb}}\label{prooflemmaimproved}
  In this section we prove Proposition \ref{lemmaimprovedb} assuming Theorem~\ref{expandprop}.

\medskip
\noindent{\bf Notation.} \quad In the proofs to follow, we shall sometimes simplify notation by symmetrizing the bilinear operator
$$
B(u,v)(t)
:=- \int_0^t e^{(t-t') \Delta} {\mathbb P} \mbox{div} \, (u \otimes v ) (t') \, dt' \, ,
$$ effectively replacing it by $B_\sigma$ defined by
$$
B_\sigma(u,v) := \tfrac 12 [B(u,v) + B(v,u)]
$$
 which is equivalent to replacing the tensor product in the definition of $B$ by
$$u\otimes_\sigma v := \tfrac 12 [u\otimes v + v \otimes u]\, .$$

\noindent Let us assume Theorem \ref{expandprop} holds, postponing its proof until the next section.  In order to use it to prove Proposition \ref{lemmaimprovedb}, we shall need the following statement.

\begin{prop}[A simple iteration]\label{propbbb}
Suppose $u_0 \in \dot B^{s_p}_{p,p}$ for some $p\in (3,\infty)$, set $T^*:= T^*(u_0)$ and define the
associated solution~$u:=NS(u_0) \in
\mathscr{L}^{1:\infty}_{p}[T<T^*]$.  Suppose that  for some~$j\in \N_0$
there exist~$p_{j}\in [1,3) $ and functions~$\tilde v_j$ and $\tilde w_j$ satisfying the statement $(S)_j$ defined by
$$(S)_j \, \, \left\{
\begin{array}{l}
\tilde v_j \in \mathscr{L}^{r:\infty}_{  p,p}(T^*)\cap \mathscr{K}_{p:\infty}(T^*)  \, , \quad r>2 \, \,\text{ s.t.} \,  \,\frac 3 p+\frac 2 r >1 \\\\
\tilde w_j \in  L^\infty (0,T^*; \dot B^{s_p}_{p,p}) \cap \mathscr{L}^\infty_{  p_j,\infty}( T^*)\quad \textrm{and}\\\\
u =  \tilde v_j + \tilde w_j \quad \textrm{in}\,\,  \,\, \mathscr{L}^{\infty}_{p}[T<T^*]\, .
\end{array} \right .
$$
Then there exists $  p_{j+1} \in [1, \tfrac 32)$ with $  p_{j+1} =1$ if $  p_j < \tfrac 32$ such that the functions~$\tilde v_{j+1}$ and $\tilde w_{j+1}$ defined by
$$
 \tilde v_{j+1}:=  \etl u_0 + B(\tilde v_j, \tilde v_j)  + 2B_\sigma (  \tilde v_j, \tilde w_j)
\quad
\textrm{and} \quad \quad  \tilde w_{j+1}:= B(\tilde w_j, \tilde w_j)
$$
satisfy $(S)_{j+1}$.  In particular, if $(S)_0$ holds, then $\tilde w_j \in \mathscr{L}^\infty_{1,\infty}(T^*)$  for all $j\geq 2$.
\end{prop}


\noindent
Postponing the proof of  Proposition \ref{propbbb} for the moment, let us proceed to prove Proposition~\ref{lemmaimprovedb}.
\\\\
{\bf Proof of Proposition~\ref{lemmaimprovedb}.} \quad
Recall $p=3\cdot 2^{k}-2$ for some integer $k\geq 2$.   Theorem \ref{expandprop} implies that~${(S)}_0$ of Proposition \ref{propbbb} is satisfied with
$$
  p_0 := \frac {6p}{2p+1}\, , \quad \tilde v_0:=   \sum_{n=0}^k u_{L,n} \, ,\quad \tilde w_0 :=   w_k \, .
$$
We can therefore  apply Proposition~\ref{propbbb} twice which gives
$$u =  \tilde v_2 + \tilde w_2$$
with
$$\tilde w_2 \in \mathscr{L}^\infty_{1,\infty}(0,T^*) \cap L^\infty(0,T^*; \bespp )
\subset  L^\infty(0,T^*; \dot B^{7/5}_{5/4,2^{k_{0}}})
\, ,$$
where~$k_{0}$ is chosen so that $2^{k_{0}}\geq 5p$. From $T^{*}<+\infty$ and  H\"older's inequality   we have
\begin{equation}
  \label{eq:12}
  \tilde w_2 \in \tl^{2^{k_0}}(0,T^*; \dot B^{{7/5}}_{ 5/4,2^{k_0}}) \subset
\tl^{2^{k_0}}(0,T^*; \dot B^{-2/5}_{5,2^{k_0}})\, .
\end{equation}
We now apply Proposition \ref{lemmaimprovedb} again, to get $u=\tilde v_{3}+\tilde w_{3}$. On the other hand, since $-2/5 < 0$ and $7/5 -2/5 = 1>0$, \eqref{eq:12} and  product estimates in Appendix \ref{paraproductsapp} therefore give
$$
\tilde w_3= B(\tilde w_2, \tilde w_2) \in \tl^{2^{k_0 -1}}(0,T^*;\dot B^{2}_{1,2^{k_0-1}}) \subset \tl^{2^{k_0-1}}(0,T^*; \dot B^{{7/5}}_{ 5/4,2^{k_0-1}})\, .$$
Applying Proposition  \ref{propbbb} and arguing as above $k_0$ times and defining~$\hat p $ by
$$
\frac1{\hat p}:= \frac 2{3p} + \frac 13
$$
and interpolating, we see that
$$\tilde w_{k_0 + 2} \in L^1(0,T^*;\dot B^{2}_{1,1}) \cap L^\infty(0,T^*;\bespp ) \subset L^3(0,T^*;\dot B^{s_{\hat p}}_{\hat p, \hat p}) \subset L^3(0,T^*;\ltrt)\, .$$
Setting $v:=  \tilde v_{k_0+2}$ and $w:=\tilde w_{k_0+2}$, Proposition  \ref{lemmaimprovedb}
follows from the above and Proposition \ref{propbbb}. \hfill $\Box$

\medskip
\noindent {\bf Proof of Proposition \ref{propbbb}. }
\quad   We start with
$$
u =  \etl u_0 + B(   \tilde v_j + \tilde w_j,   \tilde v_j + \tilde w_j)
$$
and we define
$$
 \tilde w_{j+1} := B( \tilde w_j , \tilde w_j) \, , \quad \tilde v_{j+1}: = u - \tilde w_{j+1} =  \etl u_0 +B_\sigma (\tilde w_j,   \tilde v_j)+ B(   \tilde v_j  ,   \tilde v_j  )
\, .
$$
 The fact that~$\tilde v_{j+1} \in \mathscr{L}^{r:\infty}_{  p,p}(T^*)$   is  a straightforward application of the estimates in Appendix~\ref{paraproductsapp}, so let us prove that~$\tilde v_{j+1} \in  \mathscr{K}_{p:\infty}(T^*)$.
First we notice that~$B( \tilde v_j,\tilde v_j )$ has the same properties as~$\tilde v_j$ thanks to~(\ref{estimateKr}), so we focus on~$B_{\sigma}( \tilde v_j,\tilde w_j )$.
Since~$ p_j < 3$  we have in particular $\tilde w_j  \subset L^\infty(0,T^*;L^{3,\infty})$. Then
denoting by~$G$ the gradient of the heat kernel we can write thanks to~\cite{oneil}, for any~$q$ such that~$3<q<\tilde q$ and $p\leq \tilde q <+\infty$, and defining~$1/q=1/\alpha+1/\tilde q+1/3-1$,
\begin{align*}
  \| B_{\sigma}( \tilde v_j,\tilde w_j )\|_{L^q} & \lesssim \int_{0}^{t}\| \frac 1 {(t-s)^{2}} G(\frac{\cdot}{\sqrt{t-s}}) \|_{L^{\alpha,1}} \|  \tilde v_j \otimes_\sigma \tilde w_j \|_{L^{3\tilde q/(3+\tilde q),\infty}} (s)\,ds\\
 & \lesssim \int_{0}^{t} \frac 1 {(t-s)^{2-3/(2\alpha)}} \|   \tilde v_j\|_{L^{\tilde q}} \| w\|_{L^{3,\infty}} (s)\,ds \, .
\end{align*}
 Therefore
\begin{align*}
  \| B_{\sigma}( \tilde v_j,\tilde w_j )\|_{q}  & \lesssim \int_{0}^{t} \frac 1 {(t-s)^{1-3(1/q-1/\tilde q)/2}} \frac 1 { s^{1/2-3/(2\tilde q)}} \,ds \| w\|_{\mathscr{L}^\infty_{  p_j,\infty}(T^*)} \|\tilde v_j \|_{\mathscr{K}_{p:\infty}}\\
 & \lesssim  \frac 1 {t^{1/2-3/(2q)}}  \| w\|_{\mathscr{L}^\infty_{  p_j,\infty}(T^*)}\|\tilde v_j \|_{\mathscr{K}_{p:\infty}}\,.
\end{align*}
Therefore, $B_{\sigma}(\tilde v_{j},\tilde w_{j})\in \mathscr{K}_{3:\infty}(T^*)\subset \mathscr{K}_{p:\infty}(T^*)$.
\medskip
\noindent
Now let us turn to~$\tilde w_{j+1}$.
We assume first that $3/p_{j}=1+2\eta$, with $0<2\eta<1$ (e.g. $3/2<p_{j}<3$). Set $3/q=1-\eta$, we have $w_{j}\in \mathscr{L}^\infty_{  p_j,\infty}(T^*)\subset\mathscr{L}^\infty_{  q,\infty}(T^*)$. Noticing that $s_{p_{j}}=2\eta$ and $s_{q}=-\eta$, we get $B(w_{j},w_{j})  \in \mathscr{L}^\infty_{ r,\infty}(T^*)$, with
$$
\frac 1 r =\frac 1 {p_{j}}+\frac 1 q \,,\,\,\text{ and }\,\, s_{r}=1+\eta<2\,.
$$
Next, assume that $3/p_{j}-1=1+\eta$ with $0<3\eta<1$: we still have $w_{j}\in \mathscr{L}^\infty_{  p_j,\infty}(T^*)\subset\mathscr{L}^\infty_{  q,\infty}(T^*)$, but $s_{p_{j}}=1+\eta$. Therefore, by product laws and heat estimates (see Appendix B) we get
$B(w_{j},w_{j})  \in \mathscr{L}^\infty_{ 1,\infty}(T^*)$ (notice that $s_{1}=1+1+\eta-\eta=2$) and Proposition~\ref{propbbb} is proved.    \hfill $\Box$

\subsection{The iteration procedure: proof of Theorem \ref{expandprop}}\label{iterationprocedure}
We formally define, for any vector field~$v$, the operator
\begin{equation}\label{LMopdef}
L[v](\, \cdot \, ):= {\rm{Id}} - 2B_\sigma(v, \cdot ) \, .
\end{equation}
Let us denote
$$
\mathcal{L}_{c}(X):=\{f:X \to X \, | \, f\, \textrm{is linear and bounded}\} \, .
$$
We also need to define ``source'' spaces, which correspond to spaces
where a source term for a Stokes equation would be placed. For
convenience, for a given space $X$, we denote by $\Delta X$ the
corresponding ``source'' space so that for example
$\|B(f,g)\|_X \lesssim \|\P\nabla \cdot (f\otimes g)\|_{\Delta X}$.  For our purposes, for \linebreak $1\leq a \leq b \leq \infty$ and $0<T\leq T^* \leq \infty$ we will define in view of (\ref{heatest})
$$
  \begin{aligned}
 \Delta{\mathscr L}^{a:b}_{p,q}(T)&\equiv  \Delta{\mathscr L}^{a}_{p,q}(T):= \tl^a((0,T);{\dot B^{s_{p}+2/a-2}_{p,q}}) \quad \textrm{for all}\ \ b\in [a,\infty]\, ,\\
\Delta{\mathscr L}^{a:b}_{p}(T)&:= \Delta {\mathscr L}^{a}_{p,p}(T)\, , \quad
\Delta \mathscr{L}^{a:b}_{p,q}[T<T^*]:=\bigcap_{0<T<T^*}\Delta \mathscr{L}^{a:b}_{p,q}(T)
\, .
\end{aligned}
$$
And finally, we set, for $6p/(p+3)<q\leq p$,
$$
  \begin{aligned}
    \mathscr{B}^{{1:\infty}}_{q/2}(T)& : = \{
z=\sum_{k,\text{finite}}\big (B_{\sigma}(f_k,g_k)+B_{\sigma}(\tilde f_k,\tilde g_k)\big) \, | \, f_k,g_k \in
\mathscr{L}^{1:\infty}_{q}(T)\,,\\
&\qquad \tilde f_k \in \mathscr{L}^{1:\infty}_{p}(T)
\,,\,\, \tilde g_k \in \mathscr{L}^{1:\infty}_{q/2}(T)\,, \,\text{with norm}
\,\,\|(\partial_{t}-\Delta) z\|_{\Delta \mathscr{L}^{1:\infty}_{q/2}(T)}\}\,,
\\
\mathscr{B}^{{1:\infty}}_{q/2}[T<T^{\star}]&:= \bigcap_{0<T<T^\star } \mathscr{B}^{{1:\infty}}_{q/2}(T)
\end{aligned}
$$
and
\begin{multline*}
    \mathscr{B}^{{\infty}}_{\frac{6p}{2p+1},\infty}(T) := \{
z=\sum_{k,\text{finite}} B_{\sigma}(f_k,g_k) \, | \, f_k \in
\mathscr{L}^{\infty}_{\frac{6p}{2p+1},\infty}(T)+\mathscr{L}^{1:\infty}_{p}(T)\,,\\
 g_k \in
\mathscr{L}^{\infty}_{\frac{6p}{2p+1},\infty}(T)\,, \,\text{with norm}
\,\,\| z\|_{ \mathscr{L}^{\infty}_{\frac{6p}{2p+1},\infty}(T)}\}\, .
\end{multline*}
Note that if $(\partial_{t}-\Delta)z=0$, then $z=0$ in view of its structure as a sum of Duhamel terms, justifying our choice of norm on $\mathscr{B}^{{1:\infty}}_{q/2}(T)$. Moreover, all bilinear terms in the definitions above are well-defined using the product rules and heat estimates from Appendix~\ref{paraproductsapp}, which also imply that $ \mathscr{B}^{{1:\infty}}_{q/2}(T) \hookrightarrow
\mathscr{L}^{1:\infty}_{q/2}(T)$.  Similarly, $  \mathscr{B}^{{\infty}}_{\frac{6p}{2p+1},\infty}(T) \hookrightarrow
\mathscr{L}^{\infty}_{\frac{6p}{2p+1},\infty}(T)$, hence the choice of that norm (which will later serve a different purpose from the previous one); finally, notice that~$s_{6p/(2p+1)}=1/2p>0$.
\\\\
The   following lemma, proved in Section~\ref{invertibilitysection},
shows when an operator of the form (\ref{LMopdef}) is well-defined and
satisfies a suitable a priori estimate.
\begin{lemma}[Invertibility in Besov spaces]\label{invertab}
Let  $3<p<+\infty$, $6p/(p+3)<q\leq p$,    $T>0$ and $v\in
\mathscr{L}^{1:\infty}_{p}(T)$. Then~$L[v]$ belongs to~$\mathcal{L}_{c}(  \mathscr{B}^{{1:\infty}}_{q/2}(T))
$  and is invertible,
 and we denote by~$K[v]$ its
inverse:  by construction of $K[v]$ we have the  identity
  \begin{equation}
    \label{eq:3}
    (\partial_{t}-\Delta_{v}) K[v]=\partial_{t}-\Delta\,,
  \end{equation}
with $\Delta_{v}:=\Delta-2\P\nabla \cdot(v\otimes_{\sigma}( \cdot))$.
Moreover, for $z \in \mathscr{B}^{{1:\infty}}_{\frac{6p}{2p+1}}(T)$,  if $K[v] z \in   L^{\infty}(0,T;\dot B^{s_p}_{p,\infty})$ then we have
\begin{equation}\label{estimateBrondinfini}
\|K[v]z\|_{\mathscr{B}^{{\infty}}_{\frac{6p}{2p+1},\infty}(T) }\leq C(v) \big(\|K[v]z\|_{L^{\infty}(0,T;\dot B^{s_p}_{p,\infty})}+ \|z\|_{\mathscr{B}^{{\infty}}_{\frac{6p}{2p+1},\infty}(T)}\big) \, .
\end{equation}\end{lemma}

\pagebreak

\noindent We shall now proceed to prove each part of Theorem~\ref{expandprop} separately.
\\\\
{\bf Proof of Theorem \ref{expandprop}, Part (I).}
Let us prove the following statement.
\begin{prop}\label{ajoutisa}
  Assume $u_{0}\in\dot B^{s_p}_{p,p}$, $u=NS(u_{0})\in
  \mathscr{L}^{1:\infty}_{p}[T<T^{*}]$. Then for $k$ such that $p=
  3.2^{k}-2$, there exists some $u_{L,n}\in
\mathscr{L}^{1:\infty}_{p/2^{n}}(T^{*})$  for each $0\leq n\leq k$ with
\begin{equation}\label{linearcontrol}
\|u_{L,n}\|_{ \ln{1:\infty}_{\frac{p}{2^n},\frac{p}{2^n}}(T^*)}  \lesssim C(\|u_0\|_{ \dot B^{s_p}_{p,p}}) \, ,
\end{equation}
 such that
  \begin{equation}
    \label{eq:4}
    u=\sum_{n=0}^{k} u_{L,n}+w_{k} \,\, \text{ in} \,\,  \mathscr{L}^{1:\infty}_{p}[T<T^{*}]
  \end{equation}
Moreover, $w_{k}$  satisfies, in~$\mathscr B^{1:\infty}_{p/2^{k}}
[T<T^*]\subset \mathscr L^{1:\infty}_p [T<T^*]$,
\begin{equation}
  \label{eq:5}
  w_{k}=K[u^{(k)}_{L}] \big(B(u_{L,k},u_{L,k})+2B_{\sigma}(u_{L,k},w_{k})+B(w_{k},w_{k})\big) \, ,
\end{equation}
 where  $u^{(k)}_{L}:=\sum_{n=0}^{k-1 }u_{L,n}$.
\end{prop}

\noindent {\bf Proof of Proposition \ref{ajoutisa}.} \quad  We start by writing
$$
u=u_L+ w \,  , \quad \mbox{with} \quad u_L:=e^{t\Delta}u_0 \quad
\mbox{and}\quad
w:=B(u,u) \, .
$$
Then we define  $u_{L,0}:=u_L$  which clearly satisfies the requirements of the proposition for $n=0$, and we set $u_L^{(0)}:=0$. We also set $w_0:=w$
and obviously ~$w_0\in \mathscr{B}^{1:\infty}_{p}[T<T^*] \hookrightarrow \mathscr{L}^{1:\infty}_p[T<T^*]$. We  notice
moreover that
$$
K[u_L^{(0)}  ] = L[u_L^{(0)}  ] = \rm{Id} \, ,
$$
and therefore $w_{0}$ satisfies
$$
  w_{0}=K[u^{(0)}_{L}] \big(B(u_{L,0},u_{L,0})+2B_{\sigma}(u_{L,0},w_{0})+B(w_{0},w_{0})\big)  \quad \mbox{in} \quad \mathscr{B}^{1:\infty}_{p}[T<T^*]\, .
$$
We now proceed by induction.  Assume that for some~$0 \leq n \leq k-1$,   there exist~$u_{L,j}$
in~$\mathscr{L}^{1:\infty}_{p/2^j} $ satisfying~(\ref{linearcontrol}) for all $0\leq j\leq n$ and $w_n\in \mathscr{B}^{1:\infty}_{p/2^{n}}[T<T^*]$ such that, if we set (for $n\geq 1$) $u_L^{(n)}:=\sum_{j=0}^{n-1} u_{L,j} \in   \mathscr{L}^{1:\infty}_{p}(T^*)$,
\begin{itemize}
\item we have, in $\mathscr{L}^{1:\infty}_{p}[T<T^*]$,
$$
u = \sum_{j=0}^n u_{L,j} + w_n \, ;
$$
\item  $w_n$  is such that, in $\mathscr{B}^{1:\infty}_{p/2^{n}}[T<T^*]
  \subset \mathscr{L}^{1:\infty}_{p}[T<T^*]$,
$$
w_n = K[u_L^{(n)}  ] \big (B(u_{L,n} ,u_{L,n} ) + 2B_{\sigma} (u_{L,n}
,w_n) + B(w_n,w_n) \big)\, .
$$
\end{itemize}
We then define
\begin{equation}\label{defzn}
z_n:=B(u_{L,n} ,u_{L,n} ) + B(w_n,w_n)
\end{equation}
and we notice that of course
\begin{equation}\label{Kunzn}
K[u_L^{(n)}  ]z_n = w_n -2 K[u_L^{(n)}  ]B_{\sigma} (u_{L,n} ,w_n) \, .
\end{equation}
Note that this identity makes sense because by definition~$z_n$ belongs
to~${\mathscr{B}}^{1:\infty}_{p/2^{n}}[T<T^*]$. Then by Lemma~\ref{invertab} we have
$$
\begin{aligned}
(\partial_t - \Delta) z_n & = (\partial_t - \Delta_{u_L^{(n)}})
K[u_L^{(n)}  ]z_n \\
  & =  (\partial_t - \Delta_{u_L^{(n)} + u_{L,n}}) w_n
\end{aligned}
$$
thanks to~(\ref{Kunzn}). Now we define
$$
u_L^{(n+1)}:= u_L^{(n)}  + u_{L,n} \in  \mathscr{L}^{1:\infty}_{p}
$$
and by construction we have
$$
w_n:= K[u_L^{(n+1)}] z_n \,  ,
$$
and by~(\ref{defzn}) we have
$$
K[u_L^{(n+1)}] z_n =K[u_L^{(n+1)}] \big(B(u_{L,n} ,u_{L,n} ) +
B\big(K[u_L^{(n+1)}] z_n,K[u_L^{(n+1)}] z_n\big)
\big) \, ,
$$
which actually makes sense not only in~$\mathscr B^{1:\infty}_{\frac
  p{2^{n}}}[T<T^*]$ but also  in~$\mathscr B^{1:\infty}_{\frac p{2^{n+1}}}[T<T^*]$ (notice that even at the last step, $q/2=p/2^{k}$ is such that $q/2>3p/(p+3)$). Finally setting
$$
u_{L,n+1} := K[u_L^{(n+1)}]  B(u_{L,n} ,u_{L,n} )
$$
we have that~$u_{L,n+1}  \in  \mathscr{L}^{1:\infty}_{p/2^{n+1}} $, and
defining
$$
w_{n+1}:=K[u_L^{(n+1)}] B\big(K[u_L^{(n+1)}] z_n,K[u_L^{(n+1)}] z_n\big)
$$
we have~$ w_{n+1} \in  \mathscr{B}^{1:\infty}_{p/2^{n+1}}[T<T^*] $ and
$$
K[u_L^{(n+1)}] z_n = u_{L,n+1} +w_{n+1}
$$
so finally
$$
w_{n+1}
= K[u_L^{(n+1)}  ] \big (B(u_{L,n+1} ,u_{L,n+1} ) + 2B_{\sigma}
(u_{L,n+1} ,w_{n+1}) + B(w_{n+1},w_{n+1}) \big)
  \quad \mbox{in} \quad \mathscr{B}^{1:\infty}_{p/2^{n+1}}[T<T^*] \, ,
$$
which closes the induction. As previously observed, we may iterate as long as $p/2^{n}>3p/(p+3)$, which stops when $n=k$, providing the desired decomposition. \qed

\medskip
\noindent
{\bf Proof of Theorem \ref{expandprop}, Part (II).} \quad  In view of (\ref{eq:2}), Part (II) of Theorem \ref{expandprop} is  a consequence of the next lemma: it suffices to follow the construction leading to  Part (I) above and to use  Appendix~\ref{paraproductsapp}; the details are left to the reader.

\begin{lemma}[Invertibility in Kato spaces]\label{invertkato}
Assume $T>0$, $p>3$ and $v\in \mathscr{L}^{1:\infty}_{p}(T) \cap
\ck^{1}_{p}(T)$. Then~$L[v]$ belongs to~$
\mathcal{L}_{c}(\mathscr{B}^{1:\infty}_{p/2}(T) \cap
\ck^{1}_{p}(T))$ and moreover, $L[v]$ is invertible on that space.
\end{lemma}
\noindent
We postpone the proof of this lemma to Section~\ref{proofinvertkato}.
\\\\
{\bf Proof of Theorem \ref{expandprop}, Part (III).}\quad
Part (III) of Theorem \ref{expandprop}   follows from getting an a
priori estimate on the following equation up to time $T^{*}$: recall from~(\ref{eq:5})
that
$$
 w_{k}=K[u^{(k)}_{L}] \big(B(u_{L,k},u_{L,k})+2B_{\sigma}(u_{L,k},w_{k})+B(w_{k},w_{k})\big)
  $$
 in~$\mathscr{B}^{1:\infty}_{p/2^{k} }[T<T^*] $.
\\\\
Unfortunately the continuity properties of $K[\cdot]$ are restricted: the integrability range allows only for $q/2>3p/(p+3)$, and we cannot close an estimate directly on $w_{k}$, which would require $q/2\sim 3/2$ to balance the a priori bound with a large $p$. However, if we
define $\nu_{k}$ such that $w_{k}=K[u^{(k)}_{L}] \nu_{k}$, we replace
the previous equation on $w_{k}$ by
\begin{equation}
  \label{eq:10}
   \nu_{k}= B(u_{L,k},u_{L,k})+2B_{\sigma}\big(u_{L,k},K[u^{(k)}_{L}]\nu_{k}\big)+B\big (K[u^{(k)}_{L}]\nu_{k},K[u^{(k)}_{L}]\nu_{k}\big)\, ,
\end{equation}
and the equation still holds in ~$\mathscr{B}^{1:\infty}_{p/2^{k} }[T<T^*] $.
Now, notice that $p$ was chosen such that $s_{2^{-k}p}=2/p$, as $p=3\cdot 2^{k}-2$ so
in other words, $p/2^{k}=3-2^{1-k}$. By construction, both $u_{L,k}$
and $w_{k}$ are in $\mathscr{L}^{1:\infty}_{3p/(p+1)}[T<T^*]$ (which corresponds to regularity
$s_{3p/(p+1)}=1/p$). One easily checks that $B(u_{L,k},u_{L,k})\in
\mathscr{L}^{\infty}_{\frac{3p}{2p-1} ,\infty}(T)$ (corresponding to regularity
$1-1/p$) and so do the remaining two terms on the righthand side of
\eqref{eq:10}. Therefore, we have that $\nu_{k}$ belongs to
$\mathscr{L}^{\infty}_{\frac{3p}{2p-1} ,\infty}(T)$, and we seek to
estimate this norm uniformly in $T<T^{*}$. Let us deal with the bilinear
term:
 we estimate each term of $K[u^{(k)}_{L}] \nu_{k}(=w_{k})$ in~$\mathscr{L}^{\infty}_{\frac{6p}{2p+1},\infty}$. Notice that $w_{k} \in L^{\infty}(0,T^{*}; \dot B^{s_{p}}_{p,p}) \hookrightarrow \mathscr{L}^{\infty}_{p,\infty}(T^{*})
$ from its very definition \eqref{eq:4} together with \eqref{linearcontrol}. Note also that $\frac{2p-1}{3p} < \frac{2p+1}{6p} + \frac 13$ and  recall that $s_{6p/(2p+1)}=1/2p>0$, hence we can estimate (crucially using \eqref{estimateBrondinfini} in this first step)
\begin{align*}
 \big \| B(K[u^{(k)}_{L}] \nu_{k},K[u^{(k)}_{L}]\nu_{k})\big\|_{\mathscr{L}^{\infty}_{\frac{3p}{2p-1} ,\infty}(T) } &   \lesssim
  \big\|K[u^{(k)}_{L}]\nu_{k}\big\|^{2}_{\mathscr{L}^{\infty}_{\frac{6p}{2p+1},\infty}(T)
  } \\
 &   \lesssim
  C^{2}(u^{(k)}_{L}) \big( \|\nu_{k}\|^{2}_{\mathscr{L}^{\infty}_{\frac{6p}{2p+1},\infty}(T) }+\|w_k\|^{2}_{\mathscr{L}^{\infty}_{p,\infty}(T)}\big) \\
 & \lesssim   C^{2}(u^{(k)}_{L})
 \big( \|\nu_{k}\|^{2\theta}_{\mathscr{L}^{\infty}_{\frac{3p}{2p-1} ,\infty} (T) } \big\|\nu_{k}\big\|^{2(1-\theta)}_{\mathscr{L}^{\infty}_{p,\infty}(T) } +\|w_k\|^{2}_{\mathscr{L}^{\infty}_{p,\infty}(T)}\big)
\end{align*}
with
$$
\theta:=\frac 1 2 - \frac 1 {4p-8}\,,
$$
and the important point to notice  is that~$2\theta<1$. Using the fact that
$$
\begin{aligned}
    C^{2}(u^{(k)}_{L})  \big\|\nu_{k}\big\|_{\mathscr{L}^{\infty}_{p,\infty}(T)}^{2(1-\theta)} &=    C^{2}(u^{(k)}_{L})
\big \|L[u^{(k)}_{L}] w_k\big\|_{\mathscr{L}^{\infty}_{p,\infty}(T)}^{2(1-\theta)} \\
& \lesssim
    C^{2}(u^{(k)}_{L})   C\big(\|u_{0}\|_{\dot B^{s_{p}}_{p,p}},\|u\|_{L^{\infty} (0,T^*;\dot B^{s_p}_{p,p})}\big)^{2(1-\theta)}=:M^{2(1-\theta)}\,,
\end{aligned}
$$
which can readily be seen from the definition of~$w_{k}$, we find
$$
 \big \| B(K[u^{(k)}_{L}] \nu_{k},K[u^{(k)}_{L}]\nu_{k})\big\|_{\mathscr{L}^{\infty}_{\frac{3p}{2p-1} ,\infty}(T)  } \lesssim
\| \nu_{k} \|^{2\theta}_{\mathscr{L}^{\infty}_{\frac{3p}{2p-1}
    ,\infty}(T)   } M^{2(1-\theta)}
  +C^{2}(u_{L}^{(k)})\|w_k\|^{2}_{\mathscr{L}^{\infty}_{p,\infty}(T)}\, .
$$
 The cross term is easier to
deal with, as obviously, there is only one factor $\nu_{k}$.
 Therefore we finally get
\begin{multline}
\| \nu_{k}\|_{\mathscr{L}^{\infty}_{\frac{3p}{2p-1} ,\infty}(T)   } \lesssim     \|u_{L,k}\|^{2}_{\mathscr{L}^{\infty}_{\frac{6p}{2p+1}}(T)} +C^{2}(u_{L}^{(k)})\|w_k\|^{2}_{\mathscr{L}^{\infty}_{p,\infty}(T)}\\
{}+\|u_{L,k}\|_{\mathscr{L}^{\infty}_{\frac{6p}{2p+1}}(T)}  \|\nu_{k}\|^{\theta}_{\mathscr{L}^{\infty}_{\frac{3p}{2p-1} ,\infty} (T)  } M^{1-\theta}
 +   \|\nu_{k}\|^{2\theta}_{\mathscr{L}^{\infty}_{\frac{3p}{2p-1} ,\infty}(T)   }  M^{2(1-\theta)}
 \, ,
\end{multline}
and recalling~(\ref{linearcontrol}), we find
$$
\begin{aligned}
\| \nu_{k}\|_{\mathscr{L}^{\infty}_{\frac{3p}{2p-1} ,\infty} (T)  } \lesssim     C\big(\|u_0\|_{\dot B^{s_p}_{p,p}},\|u\|_{\mathscr{L}^{\infty}_{p,\infty}(T)}\big) +    \|u_0\|_{\dot B^{s_p}_{p,p}}   \|\nu_{k}\|^{\theta}_{\mathscr{L}^{\infty}_{\frac{3p}{2p-1} ,\infty} (T)  }    M^{1-\theta}
  \\  +   \|\nu_{k}\|^{2\theta}_{\mathscr{L}^{\infty}_{\frac{3p}{2p-1} ,\infty} (T)  } M^{2(1-\theta)}
  \, .
\end{aligned}
$$
This implies that
$$
\| \nu_{k}\|_{\mathscr{L}^{\infty}_{\frac{3p}{2p-1} ,\infty}  (T) } \lesssim F\big(\|u_0\|_{\dot B^{s_p}_{p,p}}, \| u\|_{L^\infty (0,T^{*};\dot B^{s_p}_{p,p})}\big)\,,
$$
and by Sobolev embedding, we infer in particular (abusing notation by retaining the $F$ notation for a different function !)
\begin{equation}
  \label{eq:11}
  \| \nu_{k}\|_{\mathscr{B}^{\infty}_{\frac{6p}{2p+1},\infty}(T)  } \lesssim F\big(\|u_0\|_{\dot B^{s_p}_{p,p}}, \| u\|_{L^\infty (0,T^{*};\dot B^{s_p}_{p,p})}\big)\,.
\end{equation}
Finally using the fact that thanks to~(\ref{estimateBrondinfini})
$$
\|w_{k}\|_{\mathscr{B}^{\infty}_{\frac{6p}{2p+1},\infty}(T)} \lesssim  C(u^{(k)}_{L})\big(\|w\|_{L^{\infty}(0,T;\dot B^{s_p}_{p,\infty})}+ \| \nu_{k}\|_{\mathscr{B}^{\infty}_{\frac{6p}{2p+1},\infty}(T)}\big )
$$
and that the previous bound on $\nu_{k}$ holds uniformly for
$T<T^{*}$, as the righthand  side of \eqref{eq:11} does not depend on
$T$, we have   proved Part (III) of Theorem \ref{expandprop}, and the proof of the theorem is now complete. \hfill $\Box$

\section{Invertibility of ``heat flow" perturbations of identity \label{invertibilitysection}}
\subsection{Invertibility in Besov spaces: proof of  Lemma \ref{invertab}}
First we will notice that the fact that~$L[v]$ belongs to~$\mathcal{L}_{c}(
\mathscr{B}^{{1:\infty}}_{q/2}(T))$ follows from the definition of~$
\mathscr{B}^{{1:\infty}}_{q/2}(T)$, the assumption that~$v
\in\mathscr{L}^{1:\infty}_{p}(T) $, and the linear estimates from
Appendix~\ref{perturbationtheoryNS}: let~$w \in \mathscr{B}^{{1:\infty}}_{q/2}(T)\hookrightarrow \mathscr{L}^{{1:\infty}}_{q/2}(T)$ and
$$
z :=L[v] w =w-2B_\sigma (v,w) \, .
$$
By its very definition, $z$ is a finite sum of
$B_{\sigma}(\cdot,\cdot)$ with appropriate entries.
Then
  \begin{equation}
    \label{eq:7}
    \partial_{t} z -\Delta z = \partial_{t} w -\Delta w +2\mathbb{P}
\nabla\cdot (v\otimes_{\sigma} w)\,,
  \end{equation}
and by product laws recalled in Appendix~\ref{paraproductsapp} (and $\frac{6p}{p+3} <q\leq p$) we have
$\nabla\cdot(v\otimes_{\sigma} w) \in \Delta
\mathscr{L}^{{1:\infty}}_{q/2}(T)$, with appropriate norm control, since in particular $w\in
  \mathscr{L}^{{1:\infty}}_{q/2}(T)$ and
$$
\|\nabla\cdot(v\otimes_{\sigma} w)\|_{ \Delta
\mathscr{L}^{{1:\infty}}_{q/2}(T)} \lesssim  \| v\|_{ \mathscr{L}^{{1:\infty}}_{p}(T)} \| w\|_{
   \mathscr{L}^{{1:\infty}}_{q/2}(T)}\, .
$$
So
$$
 \| (\partial_{t}-\Delta) z\|_{\Delta\mathscr{L}^{{1:\infty}}_{q/2}(T)}  \lesssim  \|
(\partial_{t}-\Delta) w\|_{\Delta\mathscr{L}^{{1:\infty}}_{q/2}(T)} +
 \|\nabla\cdot(v\otimes_{\sigma} w)\|_{\Delta \mathscr{L}^{{1:\infty}}_{q/2}(T)}
 $$
which  implies that
$$
\| z\|_{ \mathscr{B}^{{1:\infty}}_{q/2}(T)}    \lesssim \|
 w\|_{ \mathscr{B}^{{1:\infty}}_{q/2}(T)} +
 \| v\|_{ \mathscr{L}^{{1:\infty}}_{p}(T)} \| w\|_{
   \mathscr{L}^{{1:\infty}}_{q/2}(T)}
$$
and hence
$$
 \| z\|_{ \mathscr{B}^{{1:\infty}}_{q/2}(T)}     \lesssim
(1+ \| v\|_{ \mathscr{L}^{{1:\infty}}_{p}(T)} )\| w\|_{ \mathscr{B}^{{1:\infty}}_{q/2}(T)} \, .
$$
Next we check  that $L[v]$ is invertible. We seek a bound on $w$ in
terms of $z$ in the equation \eqref{eq:7}. If~$z$ belongs to~$
\mathscr{B}^{{1:\infty}}_{q/2}(T)$, there are a finite number of ~$f_k,
g_k\in \mathscr{L}^{1:\infty}_{q}(T)$,
$\tilde f_{k}\in \mathscr{L}^{1:\infty}_{p}(T)$ and $\tilde g_{k}
\in\mathscr{L}^{1:\infty}_{q/2}(T)$ such that~$z =\displaystyle  \sum_k B_\sigma
(f_k,g_k)+B(\tilde f_{k},\tilde g_{k})$, which implies that
\begin{equation}
  \label{eq:8}
  \partial_t w - \Delta w  + 2\P \nabla \cdot (v \otimes_\sigma w ) = \P
\nabla \cdot (\sum_k f_k \otimes_\sigma  g_k+\tilde
f_{k}\otimes_{\sigma}\tilde g_{k}) =\partial_{t}z-\Delta z \, ,
\end{equation}
which can be solved thanks to the results of Appendix~\ref{perturbationtheoryNS}:  according to Proposition~\ref{pdr1lin}, we have the estimate
\begin{equation}\label{curlybappref}
\|w\|_{\cl^{1:\infty}_{q/2}(T)} \leq F(\|v\|_{\cl^{1:\infty}_p(T)}) \|
\nabla \cdot (\sum_k f_k\otimes_{\sigma}g_k+\tilde
f_k\otimes_{\sigma}\tilde g_k) \|_{\Delta \cl^{1:\infty}_{q/2}(T)}\, ,
\end{equation}
which we rewrite $\|w\|_{\cl^{1:\infty}_{q/2}(T)} \leq F(v) \|
z\|_{\mathscr{B}^{{1:\infty}}_{q/2}(T)}$. From there, going back to the
equation \eqref{eq:8}, we recover the structure, as
$w=z+2B_{\sigma}(v,w)$ and we may estimate (using product rules)
\begin{align*}
  \|(\partial_{t}-\Delta) w \|_{\Delta\mathscr{L}^{{1:\infty}}_{q/2}(T)}
  &\lesssim \|(\partial_{t}-\Delta) z
  \|_{\Delta\mathscr{L}^{{1:\infty}}_{q/2}(T)}+\| \nabla \cdot
  (v\otimes_{\sigma} w)\|_{\Delta \mathscr{L}^{{1:\infty}}_{q/2}(T)}\\
 \| w \|_{\mathscr{B}^{{1:\infty}}_{q/2}(T)}
  &\lesssim \| z \|_{\mathscr{B}^{{1:\infty}}_{q/2}(T)}+\|
  v\|_{\mathscr{L}^{{1:\infty}}_{p}(T)}
\| w\|_{ \mathscr{L}^{{1:\infty}}_{q/2}(T)}\\
  &\lesssim \| z \|_{\mathscr{B}^{{1:\infty}}_{q/2}(T)}(1+\tilde F(v))
\end{align*}
which is the desired bound.

\medskip
\noindent
We now proceed with proving~(\ref{estimateBrondinfini}). From the previous argument with $q/2=6p/(2p+1)$ (which is such that $q>6p/(p+3)$), we have that $w=K[v] z$ is well-defined and $w\in \mathscr{B}^{1:\infty}_{q/2}(T)$. By embedding, we know that $z\in \mathscr{B}^{{\infty}}_{\frac{6p}{2p+1},\infty}(T)$ and we assume an a priori control: $w\in L^{\infty}(0,T; \dot B^{s_{p}}_{p,\infty})$; we now seek to estimate $w$ in $\mathscr{B}^{{\infty}}_{\frac{6p}{2p+1},\infty}(T)$ in terms of $z$ in the same space and the a priori control, without using the $ \mathscr{B}^{1:\infty}_{q/2}(T)$ norm of $z$. We simply estimate (where $C$ may change from line to line)
\begin{align*}
  \| w \|_{\mathscr{L}^{{\infty}}_{\frac{6p}{2p+1},\infty}(T)}
  &\lesssim \| z
  \|_{\mathscr{L}^{{\infty}}_{\frac{6p}{2p+1},\infty}(T)}+\| \nabla \cdot
  (v\otimes_{\sigma} w)\|_{\Delta \mathscr{L}^{{\infty}}_{\frac{6p}{2p+1},\infty}(T)}\\
  &\lesssim \| z \|_{\mathscr{L}^{{\infty}}_{\frac{6p}{2p+1},\infty}(T)}+\|
  v\|_{\mathscr{L}^{{1:\infty}}_{p}(T)}
\| w\|_{ \mathscr{L}^{{\infty}}_{\frac{6p}{2p-5},\infty}(T)}\\
  &\lesssim \| z \|_{\mathscr{L}^{{\infty}}_{\frac{6p}{2p+1},\infty}(T)}+C(v)\|w\|_{\mathscr{L}^{{\infty}}_{p,\infty}(T)}^{\frac 6 {2p-5}}\| w\|^{1-\frac 6 {2p-5}}_{ \mathscr{L}^{{\infty}}_{\frac{6p}{2p+1},\infty}(T)}\, ,\\
   \| K[v] z \|_{\mathscr{L}^{{\infty}}_{\frac{6p}{2p+1},\infty}(T)}
  &\lesssim \| z
  \|_{\mathscr{L}^{{\infty}}_{\frac{6p}{2p+1},\infty}(T)}+C(v)\|w\|_{\mathscr{L}^{{\infty}}_{p,\infty}(T)}
\end{align*}
 taking good note that product rules allow an
estimate on $v\otimes_{\sigma} w$. This ends the proof of Lemma \ref{invertab}.
\subsection{Invertibility in Kato spaces: proof of Lemma \ref{invertkato}}\label{proofinvertkato}
First we know that~$L[v]$ belongs to~$
\mathcal{L}_{c}(\mathscr{B}^{1:\infty}_{p/2}(T)
)
$, let us prove that it also belongs to~$
\mathcal{L}_{c}(\mathscr{B}^{1:\infty}_{p/2}(T)
 \cap\ck^{1}_{p}(T))$.
 Since  $\ell^p \subset \ell^\infty$, it suffices to prove that
\begin{equation}\label{curlykest}
  \|B_\sigma(v,f)\|_{ \ck^{1}_{p}(T)}
\lesssim \|v\|_{\mathscr{L}^{1:\infty}_{p,\infty}(T) \cap \ck^1_p(T) } \| f\|_{\cl^{\infty}_{p,\infty}(T) \cap \ck^{1}_{p}(T)} \, .
\end{equation}
We split the time integral in the definition of $B_\sigma$, into two parts, $B^{\flat}$ from~$0$ to~$t/2$ and $B^{\sharp}$ from~$t/2$ to~$t$. On the one hand we write
$$
\begin{aligned}
 B^{\flat}(v,f)&= \int_0^{\frac t2} e^{(t-s)\D} \mathbb{P} \nabla \cdot (v( s) \otimes_\sigma f(   s))\, ds\\
 &= e^{\frac t2 \D}
  \int_0^{\frac t2} e^{(\frac t2-s)\D} \mathbb{P} \nabla \cdot (v( s) \otimes_\sigma f(   s))\, ds \\
  & = e^{\frac t2 \D} B_\sigma(v,f)(t/2)
  \end{aligned}
$$
and the usual estimates of the heat flow (see Remark \ref{usefulembedding}) along with the usual product estimates imply that
$$
 \|e^{\frac t2 \D}B_\sigma(v,f)(t/2)\|_{\mathscr{K}^{1}_p(T)}  \lesssim \|B_\sigma(v,f)(t/2)\|_{\dot B^{s_p}_{p,\infty}} \lesssim \|v\|_{\mathscr{L}^{1:\infty}_{p,\infty}(T) }  \| f\|_{ \cl^{\infty}_{p,\infty}(T)} \, .
$$
On the other hand, for some constant~$c$,
$$
2^{j} \|\Delta_{j} B^{\sharp}(v,f)\|_{ L^{p}}\lesssim \int_{t/2}^{t} 2^{2j} e^{-c(t-s)2^{2j}}
\|\Delta_{j}(v\otimes f)\|_{ L^{p}} (s) \,ds\,,
$$
and then we use, for $s>t/2$,
\begin{align*}
  \|\Delta_{j}(v \otimes_\sigma  f)(s)\|_{L^p} & \lesssim \|(v \otimes  f)(s)\|_{L^p} \lesssim \|v(s)\|_{L^\infty} \| f(s)\|_{L^p}\\
 & \lesssim \frac 1 {\sqrt s} \frac{1 }{ s^{\frac 1 2 - \frac 3 {2p}}}
 \|v\|_{\mathscr{L}^{\infty}_{p,\infty}(T) \cap \ck^1_p(T) }
 \| f\|_{L^{\infty}(0,T;\dot B^{s_p}_{p,\infty})}^{\frac{p}{2p-3}}\|f\|_{\mathscr{K}^{1}_p(T)}^{\frac{p-3}{2p-3}}\\
 & \lesssim \frac 1 { t^{ 1  - \frac 3 {2p}}}
 \|v\|_{\mathscr{L}^{\infty}_{p,\infty}(T) \cap \ck^1_p(T) }
 \| f\|_{L^{\infty}(0,T;\dot B^{s_p}_{p,\infty})}^{\frac{p}{2p-3}}\|f\|_{\mathscr{K}^{1}_p(T)}^{\frac{p-3}{2p-3}}
\end{align*}
thanks to Remark~\ref{usefulembedding}, and that provides estimate (\ref{curlykest}) and hence the expected boundedness of~$L[v]$.

\medskip
\noindent Now let us prove the invertibility.  As in the proof of Lemma \ref{invertab}, suppose
$$
z =L[v] w =w-2B_\sigma (v,w) \, .
$$
Then
$$
\partial_t z - \Delta z =\partial_t w - \Delta w + 2\P \nabla \cdot (v \otimes_\sigma w ) \, ,
$$
so we need to solve in~$\mathscr{B}^{1:\infty}_{p/2}(T)
 \cap\ck^{1}_{p}(T)$ the equation
$$
\partial_t w - \Delta w  + 2\P \nabla \cdot (v \otimes_\sigma w ) = \sum_k \P \nabla \cdot ( f_k \otimes_\sigma  g_k)  \, ,
$$
for some finite number of $f_k$ and~$g_k$ in~$\mathscr{L}^{1:\infty}_{p}(T)$.   Actually we just need to check that~$w \in \ck^{1}_{p}(T)$ due to the (\ref{curlybappref}).  We check this by first writing equivalently (we assume zero data)
$$w=2B_\sigma(v,w) +z\, ,$$
and it suffices to consider~$B_\sigma(v,w) $. This comes from the calculations leading to (\ref{curlykest}), which imply   that
$$
\begin{aligned}
\| B_\sigma(v,w) \|_{\ck^{1}_{p}(T )} &\lesssim    \|v\|_{\mathscr{L}^{1:\infty}_{p}(T) }
 \| w\|_{L^{\infty}(0,T;\dot B^{s_p}_{p,\infty})}^{\frac{p}{2p-3}}\|w\|_{\mathscr{K}^{1}_p(T)}^{\frac{p-3}{2p-3}}
 +  \|v\|_{\mathscr{L}^{1:\infty}_{p}(T) }  \| w\|_{ L^{\infty}(0,T; \dot B^{s_p}_{p,\infty}) }\end{aligned}
$$
and the result now follows from (\ref{curlybappref}).
 \appendix

\section{Estimates on linear heat equations and perturbed Navier-Stokes}\label{perturbationtheoryNS}
\newcommand\normtb[5]{\|#5\|_{\mathcal{L}^{#1}_T\dot{B}^{#2}_{#3,#4}}}
\newcommand\normb[4]{\|#4\|_{\dot{B}^{#1}_{#2,#3}}}
\newcommand\vbar{\overline{v}}

\def\virgp{\raise 2pt\hbox{,}}

\noindent In this appendix, we  state and sketch the proof of  some useful results for linear heat equations, as well as a result on the Navier-Stokes equations in $\R^d$ which
may be seen as an extension of similar results in \cite{gip3,GKP}.  Let us fix here some notation: we define
$$H(g)(t):= \int_0^t e^{(t-s)\D}\P g(s)\, ds\, , \quad s_p:= -1 + \tfrac 3p\, ,
$$
(so that $B(u,v)(t):= H(- \nabla \cdot (u \otimes v))(t)$) and we recall the notation
$$
\Delta  {\mathscr L}^r_{p,q} (0,T)\ := \mathcal{L}^r ((0,T);\dot B^{s_p -2 + \frac 2r}_{p,q})
$$
for the space
where~$g$ should belong, in order for~$H(g)$ to be in~$\cl^{r:\infty}_{p,q}(T)$:
 we recall indeed the standard heat estimates  for $r,p,q \in [1,\infty]$,  thanks to~(\ref{genheatest}),
\begin{equation}\label{heatest}
\|H(g)\|_{\cl^{r:\infty}_{p,q}(T)} \lesssim \|g\|_{\Delta  {\mathscr L}^r_{p,q}(T)}\, ,
\end{equation}
and
\begin{equation}\label{dataheatest}
\|\etl f_0\|_{\cl^{r:\infty}_{p,q}(T)} \lesssim \|f_0\|_{\dot B^{s_p}_{p,q}}\, .
\end{equation}
\ \\
Let us first study a linear equation of the type
 \begin{equation}\label{perteqlin}
 w(t) = 2B_\sigma(v,w)(t) + H(f)(t) \, .
\end{equation}
 \begin{prop}\label{pdr1lin}
Given~$3<q\leq p <\infty$, there is a
positive
non decreasing function~$F$ such that the   following holds.  Assume that~$v \in \cl^{1:\infty}_p(T)$,
$f \in \Delta  {\mathscr L}^{1:\infty}_{q/2}(T)$
for  some~$T>0$. Then there is a unique  solution $w\in   \cl^{1:\infty}_{q/2}(T)$
to~{\rm(\ref{perteqlin})}, which satisfies
$$
\|w\|_{ \cl^{1:\infty}_{q/2(T)}} \leq    \|f\|_{ \Delta  {\mathscr L}^{1}_{q/2}(T)}    F(\|v\|_{ \cl^{1:\infty}_q(T)})\, .
$$
\end{prop}
\begin{proof}
The result is rather straightforward:  by the product rules recalled in Appendix~\ref{paraproductsapp}
 one has indeed
 for all~$3<q \leq p$  and for some~$\rho>2$, for any subinterval~$(\alpha,\beta)$  of~$[0,T]$
 $$
 \|\nabla \cdot  (v \otimes_\sigma w)\|_{\Delta\mathscr{L}^{1}_{  q/2}(\alpha,\beta)} \lesssim
 \|v\|_{\mathscr{L}^{1:\rho}_{p}(\alpha,\beta)}\|w\|_{\mathscr{L}^{1:\infty}_{q/2}(\alpha,\beta)}  \, .
 $$
 This gives thanks to~(\ref{heatest})
 $$
 \|w\|_{ \cl^{1:\infty}_{q/2}(\alpha,\beta)}  \lesssim \|w(\alpha)\|_{\dot B^{s_{q/2}}_{q/2}} +  \|v\|_{\mathscr{L}^{1:\rho}_{p} (\alpha,\beta) }\|w\|_{\mathscr{L}^{1:\infty}_{q/2}(\alpha,\beta)} +   \|f \|_{ \Delta  {\mathscr L}^{1}_{q/2}(\alpha,\beta)}   \, .
 $$
 The result follows by cutting~$(0,T)$ small enough such subintervals, the number of which is an increasing function of~$ \|v\|_{\mathscr{L}^{1:\rho}_{p} (\alpha,\beta)}$.
  \end{proof}
 \medskip
 \noindent  Finally we wish to solve the equation
\begin{equation}\label{perteq}
w(t) = \etl w_0 + B(w,w)(t) + 2B_\sigma(v_1 + v_2,w)(t) + H(f_1 + f_2)(t)\end{equation}
were~$v_1,v_2,f_1$ and~$f_2$ have various regularities, adapted to the needs of this paper. The statement is the following.

\begin{prop}\label{pdr1}
Given~$q \in (3,\infty)$  and~$N \geq 1$   such that~$3(N-1) \leq q$, consider~$\delta $ in~$ (\frac3q,1)$ and define~$r$ and $r'$  by~$\tfrac1r = \tfrac Nq + \tfrac{1-\delta}2$ and $\frac 1r + \frac 1{r'}=1$.
\medskip
\noindent There is a
positive
non decreasing function of two variables~$F$ such that the   following holds.  Assume that~$v_1 \in \cl^{a:r'}_q(T)$,
for some~$1 \leq a < 2$
$v_2 \in \cl^{\frac qN}_{\frac qN}(T)$,
$f_1 \in \Delta  {\mathscr L}^{q_1}_q(T)$ and $f_2 \in \Delta  {\mathscr L}^{q_2}_q(T)$,
for  some~$T>0$, with $1\leq q_1,q_2 \leq r$. If
\begin{equation}\label{pertsmall}
\|w_0\|_{\besqq} + \|f_1\|_{\Delta  {\mathscr L}^{q_1}_q(T)} + \|f_2\|_{\Delta  {\mathscr L}^{q_2}_q(T)} \leq \Big( F \big(
\|v_1\|_{\cl^{a:r'}_q(T)}, \|v_2\|_{ \cl^{\frac qN}_{\frac qN}(T)}
\big) \Big) ^{-1}
\, ,
\end{equation}
then there is a unique  solution $w\in \mathcal{C}([0,T];\besqq)\cap \cl^{r:\infty}_q(T)$
to~{\rm(\ref{perteq})}, which satisfies
\begin{equation}\label{pertcontrol}
\|w\|_{\cl^{r:\infty}_q(T)} \leq   \left( \|w_0\|_{\besqq} + \|f_1\|_{\Delta  {\mathscr L}^{q_1}_q(T)} + \|f_2\|_{\Delta  {\mathscr L}^{q_2}_q(T)}
 \right)  F \big(
\|v_1\|_{\cl^{a:r'}_q(T)}, \|v_2\|_{ \cl^{\frac qN}_{\frac qN}(T)}
\big)\, .
\end{equation}
\end{prop}
\noindent
{\bf Proof of Proposition~\ref{pdr1}. } $ $ Note in the above that $1<r<2<r'<q<\infty$ and $r < \frac qN <
\infty$.
The proof of the proposition is rather classical, and follows for instance the methods of~\cite{gip3} (see in particular Proposition 4.1 and Theorem 3.1
of~\cite{gip3}).  We shall not give all the details of the proof but just prove the main key  estimate, namely that there exists some $K>1$ such that for any $(\alpha,\beta)\subseteq (0,T)$,
\begin{equation}\label{pertest}
\begin{aligned}
\|w\|_{\cl^{r:\infty}_q(\alpha,\beta)} & \leq K \Big( \|w(\alpha)\|_{\besqq} + \|f_1\|_{\Delta  {\mathscr L}^{q_1}_q(T)} + \|f_2\|_{\Delta  {\mathscr L}^{q_2}_q(T)} \\
 & \qquad + \big(\|w\|_{\cl^{r:r'}_q(\alpha,\beta)} +V(\alpha,\beta)\big)\|w\|_{\cl^{r:\infty}_q(\alpha,\beta)} \Big) \, ,
\end{aligned}
\end{equation}
where~$V(\alpha,\beta):= \|v_1\|_{\cl^{a:r'}_{q}(\alpha,\beta)}
+ \|v_2\|_{\cl^{\frac qN}_{\frac qN}(\alpha,\beta)}$.

\medskip
\noindent
Assuming that estimate is true, one recovers a global bound on~$w$ by  splitting $(0,T)$ into $m$
small sub-intervals where $V(\alpha,\beta) \leq \frac 1{4K}$ (since $r',\frac qN <\infty$); the size of~$m$ is tied to the function~$F$ above. One can    use the time-continuity of~$w$
to propagate (\ref{pertest}) from one subinterval to the next as long as $\|w\|_{\cl^{r:r'}_q(\alpha,\beta)} \leq \frac 1{4K}$ there as well.
 This along with assumption (\ref{pertsmall}) together imply that in fact~$\|w\|_{\cl^{r:\infty}_q(0,T)}\leq \frac 1{4K}$ which concludes the proof.

\medskip
\noindent
 To prove (\ref{pertest}), for $\e >0$, define  $r_\e$ and $\bar r_\e$
by
$$  1- \frac{d-\e}{2p} = \frac 1{r_\e}  = \frac Nq + \frac 1{\bar r_{\e}}\, \cdotp$$
Due to the assumptions on $q$, $N$ and $\d$, for a sufficiently small $\e$ we have $1\leq r_\e \leq r \leq \bar r_\e \leq 2$.
\\\\
We shall only study the terms containing~$v_1$ and~$v_2$ as the bilinear term in~$w$ is dealt with classically using the estimates of Appendix~\ref{paraproductsapp}.
Let us recall the paraproduct decomposition and the abbreviated notation
$$
fg =  {\mathcal T}_f g +  {\mathcal T}_g f +  {\mathcal R} (f,g) =:  {\mathcal T}_f g + {\mathcal R}_g f \, .
$$
Let us    set~${\hat q}:=\frac qN \leq q$. We can then write, since $r_\e \leq r$,
$$
\|B(v_2,w)\|_{\cl^{r:\infty}_q}
\leq \|H(\nabla \cdot {\mathcal R}_{v_2}w )\|_{\cl^{r_\e:\infty}_q} + \|H(\nabla \cdot {\mathcal T}_w {v_2}) \|_{\cl^{r:\infty}_{{\hat q}}} \, .
$$
Then (\ref{heatest}) gives
$$
\|B({v_2},w)\|_{\cl^{r:\infty}_q}
\lesssim \| {\mathcal R}_{v_2} w\|_{\tl^{r_\e}\dot B^{s_q -1 + \frac 2{r_\e}}_{q,q}}+ \| {\mathcal T}_w{v_2}\|_{\tl^{r}\dot B^{s_{{\hat q}} -1 + \frac 2{r}}_{{\hat q},{\hat q}}}
$$
So noticing that~$ s_q -1 + \frac 2{r_\e}=\frac \e q>0$, that~$s_\infty + \frac {2}{\hat q} <0$ and that~$s_\infty + 1-\d =-\d <0$, the product rules (\ref{prodest}) give
$$
\|B({v_2},w)\|_{\cl^{r:\infty}_q}
\lesssim\|{v_2}\|_{\mathcal{L}^{{\hat q}}\dot B^{s_\infty + \frac {2}{\hat q} }_{\infty,\infty}} \|w\|_{\mathcal{L}^{\bar r_{\e}}\dot B^{s_q + \frac 2{\bar r_\e}}_{q,q}}
+
\|{v_2}\|_{\mathcal{L}^{{\hat q}}\dot B^{s_{{\hat q}} + \frac {2}{\hat q}}_{{\hat q},{\hat q}}} \|w\|_{\mathcal{L}^{\frac 2{1-\d}}\dot B^{ -\d }_{\infty,\infty}} \, .
$$
Embeddings~(\ref{bernst}), along with the fact that~$r\leq \bar r_{\e} \leq 2$, give finally
$$
\|B(v_2,w)\|_{\cl^{r:\infty}_q}
\lesssim\|{v_2}\|_{\mathscr{L}^{{\hat q}}_{{\hat q}}} \|w\|_{\mathscr{L}^{\bar r_{\e}:\frac 2{1-\d}}_{q}}
\lesssim \|{v_2}\|_{\mathscr{L}^{{\hat q}}_{{\hat q}}} \|w\|_{\mathscr{L}^{r:\infty}_{q}}\, .
$$
Using this and a similar estimate for the other term in $2B_\sigma({v_2},w)$, the result~(\ref{pertest}) follows in view of (\ref{heatest}) and (\ref{dataheatest}) and the simpler estimate
$$\|B(v_1,w)\|_{\cl^{1:\infty}_q} \lesssim  \|v_1\|_{\mathscr{L}^{a:r'}_{q}} \|w\|_{\mathscr{L}^{r:a'}_{q}}$$
(with $\frac 1a + \frac 1{a'}=1$) which is valid since $r,a <2$.  We omit the details. \hfill $\Box$

\section{Product laws, embeddings and heat estimates \label{paraproductsapp}}

\noindent
We first recall the following standard product laws in Besov spaces, which use the theory of paraproducts.
For any distributions $f$ and $g$ which are equal as distributions to the sum of their Littlewood-Paley decompositions, we can write their product as a sum of three terms denoted as follows:
$$fg = \mathcal{T}_fg + \mathcal{T}_gf + \mathcal{R}(f,g)$$
(referred to as the low-high, high-low and high-high frequency interactions respectively), and we sometimes use the abbreviated notation
$${\mathcal R}_g f:= {\mathcal T}_g f +  {\mathcal R} (f,g)\, .$$
These terms moreover have the following properties:  for any $s_i, t_i \in \R$ and
$\bar p_i,\bar q_i , p_i, q_i,  p_i', q_i'\in [1,\infty]$
related by
$$\frac 1{\bar p_i} = \frac 1{p_i} + \frac 1{p_i'} \qquad \textrm{and} \qquad \frac 1{\bar q_i} = \frac 1{q_i} + \frac 1{q_i'}\ ,$$
we have \begin{equation}\label{prodest}
\begin{array}{c}
\|\mathcal{T}_fg\|_{\dot B^{s_1 + t_1}_{\bar{p}_1,\bar{q}_1}}
\lesssim
\|f\|_{\dot B^{s_1}_{p_1,q_1}}\|g\|_{\dot B^{t_1}_{p_1',q_1'}} \qquad \textrm{as long as}\ s_1 < 0 \quad \textrm{and}
\\
\|\mathcal{R}(f,g)\|_{\dot B^{s_2 + t_2}_{\bar{p}_2,\bar{q}_2}}
\lesssim
\|f\|_{\dot B^{s_2}_{p_2,q_2}}\|g\|_{\dot B^{t_2}_{p_2',q_2'}} \qquad \textrm{as long as}\ s_2 + t_2 > 0\ .
\end{array}
\end{equation}
That is, in the low-high or high-low interactions, the term with the low frequencies must always have a negative regularity, and in the high-high interactions the sum of the regularities must be positive.
\\\\
We now recall the following standard embedding which follows from Bernstein's inequalities,
\begin{equation}\label{bernst}
\sigma \in \R,\
1\leq p_1 \leq p_2 \leq \infty\, ,   1\leq q \leq \infty \, \, \Longrightarrow \, \, \dot B^{\sigma}_{p_1,q}(\R^3) \hookrightarrow \dot B^{\sigma - 3\left(\tfrac 1{p_1} - \tfrac 1{p_2}\right)}_{p_2,q}(\R^3)
\end{equation}
as well as the fact that $\dot B^{-1+\frac dp}_{p,q} (\R^3)
\hookrightarrow L^3(\R^3)$ if $1\leq p,q<3$ (cf., e.g., \cite{gk}),
and $\dot B^{-1+\frac 3p}_{p,\infty} (\R^3)
\hookrightarrow L^{3,\infty}(\R^3)$, where the last space stands for
the weak Lebesgue space.

\medskip

\noindent Let us also recall the following standard heat estimate.  For any $p\in [1,\infty]$, there exist some $c_0, c >0$ such that for any $f \in \mathcal{S}'$ and $j\in \Z$,
\begin{equation}\label{genheatest}
{\|\D_j(\etl f)\|_p \leq c_0 e^{-ct2^{2j}}\|\D_j f\|_p}\, .
\end{equation}
Hence for $ 0 < t  \leq \infty$, recalling
$$B (u,v)(t):= \int_{0}^t e^{(t-\tau)\D}\mathbb{P}\nabla \cdot (u(\tau)\otimes v(\tau))\, d\tau\, ,$$
H\"older's inequality implies that for any $p,q \in [1,\infty]$, $s\in \R$ and $1\leq r \leq \infty$,
\begin{equation}\label{genbilinesta}
\|B (u,v)(t)\|_{\dot B^{s+ 2 (1 - \frac 1r )}_{p,q}} \lesssim \|\mathbb{P}\nabla \cdot (u\otimes v)\|_{\tl^{r}( 0,t; \dot B^s_{p,q})}
\, ,
\end{equation}
and hence moreover by Bernstein's inequalities and the zero-order nature of $\P$,
\begin{equation}\label{genbilinestaaa}
\|B (u,v)(t)\|_{\dot B^{s+ 2 (1 - \frac 1r )}_{p,q}} \lesssim \|u\otimes v\|_{\tl^{r}(0,t; \dot B^{s+1}_{p,q})}\, .
\end{equation}
More generally
 for any $\tilde r\in [r ,\infty]$ Young's inequality for convolutions implies
\begin{equation}\label{genbilinest}
\|B (u,v)\|_{\tl^{\tilde r} (0,t; \dot B^{s+2 + 2 (\frac 1{\tilde r} - \frac 1r )}_{p,q}  )} \lesssim  \|u\otimes v\|_{\tl^{r}(0,t; \dot B^{s+1}_{p,q})}
\end{equation}
which can be combined with the product laws and embeddings above to give various ``bilinear estimates".

\medskip
\noindent
We shall also need   estimates on the bilinear form in Kato-type spaces $\ck_p(T)$ (cf.~(\ref{katospacedef})), which can be obtained from  standard estimates for the linear Stokes kernel (see, e.g., \cite{Planchon}): there exists a universal constant $c >0$  such that for any~$T >0$ and any $p,q,r \in [1,\infty]$ such that
$$\displaystyle{0 <  \frac 1{p} + \frac 1{q} < \frac 13 + \frac 1{r} \quad \textrm{and} \quad \frac 1{r} \leq \frac 1{p} + \frac 1{q} \leq 1 \, , }$$
for any $f$ and $g$ one has
\begin{equation}\label{estimateKr}
\begin{array}{c}
\displaystyle{\|B (f,g)\|_{\mathscr{K}_{r}(T)} \leq c[(\tfrac 1p + \tfrac 1q)^{-1} + (\tfrac 13 + \tfrac 1r - \tfrac 1p - \tfrac 1q)^{-1}] \|f\|_{\mathscr{K}_{p}(T)}\|g\|_{\mathscr{K}_{q}(T)} } \, .
\end{array}
\end{equation}
We also recall the relationship between Kato and negative-regularity Besov spaces: for any $p>3$, there exists some $c=c(p)>0$ such that for any $v_0 \in \dot B^{-1+\frac 3p}_{p,\infty}$,
\begin{equation}\label{heatbesov}
c^{-1} \|v_0\|_{\dot B^{-1+\frac 3p}_{p,\infty}} \leq  \|\etl v_0\|_{\mathscr{K}_p(0,\infty)} \leq c \|v_0\|_{\dot B^{-1+\frac 3p}_{p,\infty}}   \, ,
\end{equation}
which in particular gives an estimate on heat flows with initial data
in such Besov spaces.  For a reference for all of the above, see
\cite{Planchon,gip3} and references therein.

\def\ocirc#1{\ifmmode\setbox0=\hbox{$#1$}\dimen0=\ht0 \advance\dimen0
  by1pt\rlap{\hbox to\wd0{\hss\raise\dimen0
  \hbox{\hskip.2em$\scriptscriptstyle\circ$}\hss}}#1\else {\accent"17 #1}\fi}
  \def\polhk#1{\setbox0=\hbox{#1}{\ooalign{\hidewidth
  \lower1.5ex\hbox{`}\hidewidth\crcr\unhbox0}}}

\end{document}